\documentclass[12pt]{amsart}

\usepackage{amsmath, amsthm, amssymb}
\input xy
\xyoption{all}
\usepackage{mathrsfs}
\usepackage{enumerate}
\usepackage{color}
\usepackage{tikz}
\usetikzlibrary{matrix}
\usepackage{caption}
\usepackage{subcaption}

\newtheorem{dummy}{}[section]
\newtheorem{theorem}[dummy]{Theorem}
\newtheorem{proposition}[dummy]{Proposition}
\newtheorem{lemma}[dummy]{Lemma}
\newtheorem{corollary}[dummy]{Corollary}

\theoremstyle{definition}
\newtheorem{definition}[dummy]{Definition}

\newtheorem{remark}[dummy]{Remark}

\newcommand{\Z}{\ensuremath{\mathbb{Z}}}
\newcommand{\Q}{\ensuremath{\mathbb{Q}}}
\newcommand{\C}{\ensuremath{\mathbb{C}}}
\renewcommand{\P}{\ensuremath{\mathbb{P}}}

\newcommand{\M}{\ensuremath{\overline{\mathcal{M}}}}

\renewcommand{\O}{\ensuremath{\mathcal{O}}}

\newcommand{\ev}{\mathrm{ev}}
\newcommand{\vir}{\mathrm{vir}}

\newcommand{\one}{\ensuremath{\mathbf{1}}}

\newcommand{\bbeta}{\ensuremath{\vec{\beta}}}

\newcommand{\GIT}{\ensuremath{\mathbin{/\mkern-6mu/}}}

\renewcommand{\H}{\ensuremath{\mathcal{H}}}

\newcommand{\mult}{\ensuremath{\textrm{mult}}}
\renewcommand{\i}{\ensuremath{\textrm{i}}}

\newcommand{\ct}{\mathrm{ct}}

\usepackage{geometry}
\usepackage{color}

\begin{document}

\title{Wall-crossing in genus-zero hybrid theory}

\begin{abstract}
The hybrid model is the Landau--Ginzburg-type theory that is expected, via the Landau-Ginzburg/Calabi-Yau correspondence, to match the Gromov--Witten theory of a complete intersection in weighted projective space.  We prove a wall-crossing formula exhibiting the dependence of the genus-zero hybrid model on its stability parameter, generalizing the work of \cite{RR} for quantum singularity theory and paralleling the work of Ciocan-Fontanine--Kim \cite{CFKZero} for quasimaps.  This completes the proof of the genus-zero Landau-Ginzburg/Calabi-Yau correspondence for compete intersections of hypersurfaces of the same degree, as well as the proof of the all-genus hybrid wall-crossing \cite{CJRGLSM}.
\end{abstract}

\author[E.~Clader]{Emily Clader}
\address{Department of Mathematics, San Francisco State University, 1600 Holloway Avenue, San Francisco, CA 94132, USA}
\email{eclader@sfsu.edu}

\author[D.~Ross]{Dustin Ross}
\address{Department of Mathematics, San Francisco State University, 1600 Holloway Avenue, San Francisco, CA 94132, USA}
\email{rossd@sfsu.edu}

\maketitle

\section{Introduction}

The gauged linear sigma model (GLSM) has been the subject of intense study by both mathematicians and physicists since its introduction by Witten \cite{Witten} in the 1990s \cite{FJRGLSM, CJRGLSM, HIV, HV}.  Special cases of the GLSM include the Gromov--Witten theory---or, more generally, the quasimap theory---of nonsingular complete intersections in GIT quotients, as well as the Fan--Jarvis--Ruan--Witten (FJRW) theory of nondegenerate singularities.  In particular, the GLSM provides an ideal context in which to understand the Landau--Ginzburg/Calabi--Yau (LG/CY) correspondence relating the Gromov--Witten theory of a nonsingular hypersurface in weighted projective space to the FJRW theory of its defining polynomial; the relationship between these two theories, from the GLSM perspective, is encoded in a variation of GIT on the target geometry.

More precisely, the GLSM depends on the choice of a GIT quotient $X_{\theta} = [V \GIT_{\!\theta} G]$ equipped with a polynomial function $W: X_{\theta} \rightarrow \C$, and a stability parameter $\epsilon \in \Q^+$.  Suppose we take the GIT quotient to be
\[X_+ := \O_{\P(w_1, \ldots, w_M)}(-d) = (\C^M \times \C) \GIT_{\!\theta} \C^*,\]
where $\C^*$ acts with weights $(w_1, \ldots, w_M, -d)$ and $\theta \in \text{Hom}_{\Z}(\C^*,\C^*) \cong \Z$ is any positive character, and let
\[W(x_1, \ldots, x_M, p) = pF(x_1, \ldots, x_M)\]
for a nondegenerate quasihomogeneous polynomial $F \in \C[x_1, \ldots, x_M]$ of weights $w_1, \ldots, w_M$ and degree $d$.  Then the GLSM recovers the Gromov--Witten theory of the hypersurface $\{F = 0\} \subseteq \P(w_1, \ldots, w_M)$ when $\epsilon \gg 0$, while for smaller $\epsilon$ it coincides with the quasimap theory developed by Ciocan-Fontanine--Kim--Maulik \cite{CFKM, CFK3, CFKZero, CFKHigher}.  The passage from $\epsilon \gg 0$ to the asymptotic stability condition $\epsilon = 0+$ can be viewed as a manifestation of mirror symmetry; in particular, a generating function of genus-zero invariants for $\epsilon = 0+$ is precisely Givental's $I$-function.  Ciocan-Fontanine and Kim gave a new proof of the genus-zero mirror theorem \cite{CFKZero} by demonstrating a strikingly simple wall-crossing formula that encodes how the genus-zero quasimap invariants change with $\epsilon$.

On the other hand, taking a negative character of $\C^*$ in the above quotient yields
\[X_- := [\C^M/\Z_d],\]
where $\Z_d$ acts diagonally with weights $(w_1, \ldots, w_M)$.  The resulting GLSM is the FJRW theory of the polynomial $F$ when $\epsilon \gg 0$, and for smaller $\epsilon$ it recovers the quantum singularity theory studied by Ruan and the second author in \cite{RR}.  The analogous analysis to the above was carried out in this chamber in \cite{RR}, yielding genus-zero wall-crossing formulas for the dependence of the theory on $\epsilon$ and a new proof of the genus-zero Landau--Ginzburg mirror theorem.  From here, the genus-zero LG/CY correspondence follows by relating the $I$-functions of Gromov--Witten and FJRW theory, a rather delicate process involving analytic continuation that was proven by Chiodo--Iritani--Ruan \cite{CR, CIR}.

Two natural questions arise from this perspective on the LG/CY correspondence.  First, can it be adapted to gauged linear sigma models associated to other GIT quotients?  And second, can it be generalized to higher genus?  

In particular, replacing the hypersurface $\{F= 0\} \subseteq \P(w_1, \ldots, w_M)$ with a nonsingular complete intersection $Y = \{F_1 = \cdots = F_N = 0\}$ of degrees $d_1, \ldots, d_N$ corresponds to considering a GIT quotient
\[(\C^M \times \C^N) \GIT_{\!\theta} \C^*,\]
in which $\C^*$ acts with weights $(w_1, \ldots, w_M, -d_1, \ldots, -d_N)$.  The GLSM associated to this quotient with a positive character coincides with the Gromov--Witten (or quasimap) theory of $Y$.  In order to ensure the properness of the GLSM moduli space in the negative chamber, however, one must assume that $d_1 = \cdots = d_N$, as this implies that the theory admits a ``good lift" \cite{FJRGLSM}.  Under this assumption, the GLSM for a negative character is known in the physics literature as the ``hybrid model" and was studied mathematically by the first author in \cite{Clader}.   It is a curve-counting theory over a moduli space $Z^{\epsilon}_{g,n,\beta}$ parameterizing genus-$g$ marked orbifold curves $(C; q_1, \ldots, q_n)$ together with a degree-$\beta$ line bundle $L$ and a section
\[\vec{p} \in \Gamma \left((L^{\otimes - d} \otimes \omega_{\log})^{\oplus N}\right)\]
with vanishing order at most $1/\epsilon$.

The genus-zero wall-crossing for the quasimap theory of $Y$ was carried out by Ciocan-Fontanine--Kim in \cite{CFKZero}, while the analytic continuation relating $\epsilon = 0+$ quasimap theory to $\epsilon = 0+$ hybrid theory was done---under a Calabi--Yau hypothesis---in our previous work \cite{ClaRo}.  The first theorem of the current paper, which states the genus-zero wall-crossing for hybrid theory, is the natural conclusion of that story:

\begin{theorem}
\label{thm:1}
Let $Y \subseteq \P(w_1, \ldots, w_M)$ be a nonsingular complete intersection defined by the vanishing of a collection of polynomials of degree $d$, where $w_i|d$ for all $i$. The $J$-functions of $\epsilon$-stable and $\infty$-stable hybrid theory are related by
\[J^{\epsilon}(q,z) = J^{\infty}(q,z \mathbf{\one} + [J^{\epsilon}]_+(q,-z),z),\]
where $[J^{\epsilon}]_+$ is the part of $J^{\epsilon}$ with non-negative powers of $z$.
\end{theorem}

See Section~\ref{sec:J} below for the precise definitions of the $J$-functions, which are generating functions of $\epsilon$-stable hybrid invariants.  In particular, the conjunction of \cite{CFKZero}, \cite{ClaRo}, and Theorem~\ref{thm:1} verifies the genus-zero Landau--Ginzburg/Calabi--Yau correspondence for all nonsingular Calabi--Yau complete intersections $Y \subseteq \P(w_1, \ldots, w_M)$ such that $w_i|d$ for all $i$.

We also extend the methods of \cite{RR} to prove a stronger wall-crossing statement, on the level not only of invariants but of virtual fundamental classes.  The statement involves comparison maps
\[c: Z^{\infty}_{0,n,\beta} \rightarrow Z^{\epsilon}_{0,n,\beta}\]
and
\[b_{\vec{\beta}} = b_{(\beta_1, \ldots, \beta_k)}: Z^{\epsilon}_{0,n+k,\beta_0} \rightarrow Z^{\epsilon}_{0,n, \beta_0 + \sum_i \beta_i}\]
whose definitions appear in Section~\ref{subsec:bc}.  The theorem, with this notation established, is the following:

\begin{theorem}
\label{thm:2}
Let $Y \subseteq \P(w_1, \ldots, w_M)$ be as in Theorem~\ref{thm:1}.  Then
\begin{equation}
\label{eq:mainWC}
\sum_{\beta} q^\beta  [Z^{\epsilon}_{0,n,\beta}]^{\vir} =\sum_{\beta_0, \beta_1, \ldots, \beta_k} \frac{q^{\beta_0}}{k!} b_{\bbeta*}c_* \left(\prod_{i = 1}^k q^{\beta_i} \ev_{n + i}^*(\mu_{\beta_i}^\epsilon(-\psi_{n + i})) \cap [Z^{\infty}_{0,n+k,\beta_0}]^{\vir} \right),
\end{equation}
where the sums are over all degrees for which the above moduli spaces are nonempty and $\mu^{\epsilon}_{\beta}(z)$ denotes the coefficient of $q^{\beta}$ in $-z\one + [J^{\epsilon}]_+(q,z)$.
\end{theorem}

The form of Theorem~\ref{thm:2} is identical to the higher-genus wall-crossing statement for the hybrid model proven by Janda, Ruan, and the first author in \cite{CJRGLSM}.  However, the proof of the higher-genus statement is an induction in which the genus-zero base case must be proven independently.  Thus, the proof in \cite{CJRGLSM} in fact relies on Theorem~\ref{thm:2}, so this work also completes the verification of higher-genus wall-crossing in the hybrid model.

It should be noted that the higher-genus LG/CY correspondence for the hybrid model still remains conjectural.  Indeed, although wall-crossing statements have now been established in both the hybrid phase and the quasimap phase (the latter by Ciocan-Fontanine--Kim in \cite{CFKHigher}, or by the alternative proof of \cite{CJR}), the analytic continuation relating the $\epsilon=0+$ theories on the two sides is a subtle issue that has so far been tackled only in genus one, by Guo and the second author \cite{GR2}.

\subsection{Plan of the paper}

We begin, in Section~\ref{sec:definitions}, by reviewing the definition of the hybrid model, including the state space, the moduli space, the genus-zero virtual cycle and correlators, and the $J$-function.  In Section~\ref{sec:localization}, we introduce an action of the torus $\mathbf{T} = (\C^*)^N$ on the moduli space by scaling the section $\vec{p}$, which yields a $\mathbf{T}$-equivariant virtual cycle, and we carefully analyze the contributions to the virtual cycle from each $\mathbf{T}$-fixed locus.  In particular, the fixed loci are indexed by decorated graphs whose vertices correspond to moduli spaces of weighted spin curves.  Mimicking and generalizing the techniques of \cite{RR}, we prove in Section~\ref{sec:vertexWC} the local analogues of Theorems~\ref{thm:1} and \ref{thm:2} at each vertex.  Finally, in Section~\ref{sec:proofs}, we prove Theorem~\ref{thm:1} via the vertex wall-crossing together with a localization recursion, and we deduce Theorem~\ref{thm:2} by localization on both sides.

\subsection{Acknowledgments}

The authors thank Felix Janda and Yongbin Ruan for many enlightening conversations.  This work was completed with the support of a Development of Research and Creativity grant from San Francisco State University in Spring 2018, during which the first author was a Research Member at the Mathematical Sciences Research Institute.

\section{Definitions and set-up}
\label{sec:definitions}

We review the definition of the hybrid model, which is a special case of the more general gauged linear sigma model (GLSM) constructed by Fan, Jarvis, and Ruan \cite{FJRGLSM}.

Let $F_1(x_1, \ldots, x_M), \ldots, F_N(x_1, \ldots, x_M)$ be quasihomogeneous polynomials of the same weights $w_1, \ldots, w_M$ and the same degree $d$, defining a nonsingular complete intersection
\[Y:=\{F_1 = \cdots = F_N = 0\} \subseteq \P(w_1, \ldots, w_M).\]
Assume, furthermore, that $w_i | d$ for each $i$.

The general GLSM depends on the choice of a GIT quotient $X=[V\GIT_{\!\theta}G]$, a polynomial function $W: X \rightarrow \C$ known as the superpotential, and an action of $\C^*$ on $V$ known as the $R$-charge.  In our case, $V = \C^{M+N}$ with coordinates $(x_1, \ldots, x_M, p_1, \ldots, p_N)$, and
\[G:= \{(g^{w_1}, \ldots, g^{w_M}, g^{-d}, \ldots, g^{-d}) \; | \; g \in \C^*\} \cong \C^*\]
acts diagonally on $V$.  For any negative character $\theta \in \text{Hom}_{\Z}(\C^*,\C^*) \cong \Z$, the resulting GIT quotient is
\[X = \bigoplus_{i=1}^M \O_{\P(d,\ldots, d)}(-w_i).\]
The superpotential on this space is defined by
\[W(x_1, \ldots, x_M, p_1, \ldots, p_N) := \sum_{j=1}^N p_jF_j(x_1, \ldots, x_M),\]
and the $R$-charge acts by diagonal multiplication on the $p$-coordinates.  The critical locus of $W$, i.e. the points where $dW=0$, is the zero section
\[Z:= \P(d,\ldots, d) \subseteq X,\]
as one readily checks from the fact that $Y$ is nonsingular.

\subsection{State space}

In what follows, insertions to hybrid model correlators are chosen from the space
\[\widetilde{\mathcal{H}} := H^*_{\text{CR}}(X).\]
This is not precisely the state space of the GLSM, but it maps surjectively to the ``compact-type" part of the GLSM state space; see \cite[Section 2.1]{CJRGLSM} and Remark~\ref{rem:ct} below for further discussion.

The space $\widetilde{\mathcal{H}}$ decomposes into summands indexed by the components of the inertia stack $\mathcal{I}X$, which are labeled by elements of $g \in G$ with nonempty fixed locus.  More specifically, the elements $(g^{w_1}, \ldots, g^{w_M}, g^{-d}, \ldots, g^{-d}) \in G$ with nonempty fixed locus are those for which $g^d =1$, and for such $g$, it is straightforward to check that the fixed locus is
\[X_g:= \bigoplus_{i \in F_g} \O_{\P(\vec{d})}(-w_i) \subseteq X,\]
where
\[F_g:= \{i \; | \; g^{w_i} = 1\} \subseteq \{1, \ldots, M\}.\]
Thus, we have
\[\widetilde{\mathcal{H}} = \bigoplus_{g \in \Z_d} H^*(X_g).\]

\subsection{Moduli space}

The general definition of the moduli space in the GLSM was proposed by Fan--Jarvis--Ruan \cite{FJRGLSM}, building on the notion of quasimaps introduced by Ciocan-Fontanine and Kim.

Fix a genus $g$, a degree $\beta \in \Z$, a nonnegative integer $n$, and a positive rational number $\epsilon$.

\begin{definition}
\label{def:qmap}
An $\epsilon$-stable Landau--Ginzburg quasimap to $Z$ consists of an $n$-pointed prestable orbifold curve $(C;q_1, \ldots, q_n)$ of genus $g$ with nontrivial isotropy only at marked points and nodes, an orbifold line bundle $L$ on $C$, and a section
\[\vec{p} = (p_1, \ldots, p_N) \in \Gamma((L^{\otimes -d} \otimes \omega_{\log})^{\oplus N}),\]
where
\[\omega_{\log}:= \omega_C([\overline{q}_1] + \cdots + [\overline{q}_N])\]
for the coarse divisors $[\overline{q}_i]$, satisfying the following conditions:
\begin{itemize}
\item {\it Representability}: For every $q \in C$ with isotropy group $G_q$, the homomorphism $G_q \rightarrow \C^*$ giving the action of the isotropy group on the bundle $L$ is injective.
\item {\it Nondegeneracy}: The zero set of $\vec{p}$ is finite and disjoint from the marked points and nodes of $C$, and for each zero $q$ of $\vec{p}$, the order of the zero (that is, the common order of vanishing of $p_1, \ldots, p_N$) satisfies
\[\text{ord}_q(\vec{p}) \leq \frac{1}{\epsilon}.\]
\item {\it Stability}: The $\Q$-line bundle
\[(L^{\otimes -d} \otimes \omega_{\log})^{\otimes \epsilon} \otimes \omega_{\log}\]
is ample.
\end{itemize}
The zeroes of $\vec{p}$ are referred to as {\it basepoints} of the quasimap, and the {\it degree} of the quasimap is defined as
\[\beta:= \text{deg}(L^{\otimes -d} \otimes \omega_{\log}).\]
Note that $\beta$ must be an integer, since if $L^{\otimes -d} \otimes \omega_{\log}$ had nontrivial orbifold structure then basepoints would be forced to occur at special points.
\end{definition}

Fan--Jarvis--Ruan proved in \cite{FJRGLSM} that there is a proper, separated Deligne--Mumford stack $Z^{\epsilon}_{g,n,\beta}$ parameterizing genus-$g$, $n$-pointed, $\epsilon$-stable Landau--Ginzburg maps of degree $\beta$ to $Z$, up to the natural notion of isomorphism.

\subsection{Multiplicities and evaluation maps}

Recall that if $q$ is a point on an orbifold curve $C$ with isotropy group $\Z_r$ and $L$ is an orbifold line bundle on $C$, then the {\it multiplicity} of $L$ at $q$ is defined as the number $m \in \Q/\Z$ such that the canonical generator of $\Z_r$ acts on the total space of $L$ in local coordinates by'
\[(x,v) \mapsto \left(e^{2\pi \i \frac{1}{r}} x, e^{2\pi\i m}v\right).\]
In our case, all multiplicities can be taken to lie in the set $\left\{0,\frac{1}{d}, \ldots, \frac{d-1}{d}\right\}$.  For a tuple $\vec{m} = (m_1, \ldots, m_n)$ of such multiplicities, we define
\[Z^{\epsilon}_{g,\vec{m},\beta} \subseteq Z^{\epsilon}_{g,n,\beta}\]
as the (open and closed) substack consisting of Landau--Ginzburg quasimaps for which the multiplicity of $L$ at $q_i$ is $m_i$.

It is a straightforward exercise (see, for example, \cite[Section 2.3]{CJRGLSM}) to check that $Z^{\epsilon}_{g,\vec{m},\beta}$ is nonempty only if
\begin{equation}
\label{eq:mults}
\frac{-\beta + 2g-2+n}{d} - \sum_{i=1}^n m_i \in \Z.
\end{equation}
In particular, since \eqref{eq:mults} is independent of the $i$th marked point if and only if $m_i = \frac{1}{d}$, this is the only case in which there is a forgetful map on $Z^{\epsilon}_{g,\vec{m},\beta}$ forgetting $q_i$ and its orbifold structure.

To define evaluation maps
\[\ev_k: Z^{\epsilon}_{g,n,\beta} \rightarrow \overline{\mathcal{I}}Z\subseteq \overline{\mathcal{I}}X = \bigsqcup_{g \in \Z_d} \bigoplus_{i \in F_g}\O_{\P^{N-1}}(-w_i)\]
to the rigidified inertia stack of $Z$, let $\pi: \mathcal{C} \rightarrow Z^{\epsilon}_{g,n,\beta}$ be the universal curve, let $\mathcal{L}$ be the universal bundle, and let $\vec{\rho}$ be the universal section of the bundle $(\mathcal{L}^{\otimes -d} \otimes \omega_{\pi,\log})^{\oplus N}$.  If $\Delta_k \subseteq \mathcal{C}$ denotes the divisor corresponding to the $k$th orbifold marked point, then
\[\vec{\rho}\big|_{\Delta_k} \in \Gamma \left( (\mathcal{L}^{\otimes -d})^{\oplus N} \bigg|_{\Delta_k} \right),\]
using the fact that $\omega_{\pi,\log}|_{\Delta_k}$ is trivial.  Thus, evaluating $\vec{\rho}|_{\Delta_k}$ at the fiber over a point $(C;q_1, \ldots, q_n; L; \vec{p})$ in the moduli space yields an element of $\P^{N-1}$, and by definition, $\ev_k$ sends $Z^{\epsilon}_{g,\vec{m},\beta}$ to the copy of $\P^{N-1}$ sitting inside $\overline{\mathcal{I}}X$ as the zero section in the sector indexed by $g = e^{2\pi \i m_k} \in \Z_d$.

\subsection{Comparison maps}
\label{subsec:bc}

There are two types of comparison maps that relate the hybrid moduli spaces to one another.  Careful definitions appear in \cite[Section 3.2]{CJRGLSM}, based on the ideas of \cite[Section 3.2]{CFKZero}.  First, if $\vec{\beta} = (\beta_1, \ldots, \beta_k)$ is a tuple of nonnegative integers and $\vec{m} = (m_1, \ldots, m_k)$ is defined by $m_i := \left\langle \frac{\beta_i+1}{d}\right\rangle$ for each $i$, then the morphism
\[b_{\vec\beta}: Z^{\epsilon}_{g,n+\vec{m}, \beta - \sum_{i=1}^k \beta_i} \rightarrow Z^{\epsilon}_{g,n,\beta}\]
replaces the last $k$ marked points with basepoints of orders $\beta_1, \ldots, \beta_k$ and contracts unstable components (replacing them by basepoints) as necessary.

Similarly, the comparison map
\[c: Z^{\infty}_{g,n,\beta} \rightarrow Z^{\epsilon}_{g,n,\beta}\]
contracts any rational tails that become unstable under the change of stability condition and replaces them with basepoints.

\subsection{Virtual cycle}

The general definition of the virtual cycle
\[[Z^{\epsilon}_{g,n,\beta}]^{\vir} \in A_*(Z^{\epsilon}_{g,n,\beta})\]
proceeds by the cosection technique of Kiem--Li \cite{KL}, following closely related work of Chang--Li \cite{CL} and Chang--Li--Li \cite{CLL}.  In genus zero, however, the situation is substantially simpler: the condition that $w_i|d$ implies that 
\[R^1\pi_*\left( \bigoplus_{i=1}^M \left(\mathcal{L}^{\otimes w_i} \otimes \O\left(-\sum_{k=1}^n \Delta_k\right)\right)\right)\]
is a vector bundle (see \cite[Section 4.2.9]{Clader}, or the analogous argument in Lemma~\ref{lem:vanishing} below), and we have

\[[Z^{\epsilon}_{0,n,\beta}]^{\vir} = e \left(\bigoplus_{i=1}^M R^1\pi_*\left( \mathcal{L}^{\otimes w_i} \otimes \O\left(-\sum_{k=1}^n \Delta_k\right)\right)\right) \cap [Z^{\epsilon}_{0,n,\beta}].\]
In fact, twisting down by the orbifold marked point is equivalent, on coarse underlying curves and hence on cohomology, to not twisting down at all if the multiplicity of $L^{\otimes w_i}$ is nonzero, or to twisting down by the coarse divisor $[\overline{q}_k]$ if the multiplicity is zero.  Thus,
\[[Z^{\epsilon}_{0,\vec{m},\beta}]^{\vir} = e \left(\bigoplus_{i=1}^M R^1\pi_*\left( \mathcal{L}^{\otimes w_i} \right) \otimes \O\left(-\sum_{k \; | \; w_i m_k \in \Z} \overline{\Delta}_k\right)\right) \cap [Z^{\epsilon}_{0,\vec{m},\beta}],\]
where $\overline{\Delta}_k$ is the marked point divisor pulled back from the coarse underlying curve.

Equipped with this virtual cycle, we can define correlators in the hybrid model.  Recall that the psi classes are defined by
\[\psi_k = c_1(\mathbb{L}_k) \in A^*(Z^{\epsilon}_{g,n,\beta})\]
for each $k\in \{1,\ldots, n\}$, where $\mathbb{L}_k$ is the line bundle whose fiber over a moduli point is the cotangent line to the coarse curve at the $k$th marked point.

\begin{definition}
Given
\[\phi_1, \ldots, \phi_n \in \widetilde{\mathcal{H}}\]
and nonnegative integers $a_1, \ldots, a_n$, the associated {\it genus-0, degree-$\beta$, $\epsilon$-stable GLSM correlator} is
\[\langle \phi_1 \psi^{a_1} \cdots \phi_n \psi^{a_n} \rangle^{\epsilon}_{0,n,\beta} = \int_{[Z^{\epsilon}_{0,n,\beta}]^{\vir}} \ev_1^*(\phi_1)\psi_1^{a_1} \cdots \ev_n^*(\phi_n)\psi_n^{a_n}.\]
\end{definition}

\subsection{Unit and pairing}
\label{sec:pairing}

Given the discussion following equation \eqref{eq:mults}, the role of the unit in the GLSM theory is played by $\one := \one_{(1/d)}$, the fundamental class in the twisted sector $H^*(X_{e^{2\pi\i/d}}) \subseteq \widetilde{\mathcal{H}}$.

Using this, we define a pairing on the state space $\widetilde{\mathcal{H}}$ by
\[\left( \phi_1, \phi_2 \right) := \langle \phi_1 \; \phi_2 \; \one \rangle^{\epsilon}_{0,3,0}.\]
More explicitly, for each $m \in \Q/\Z$ and each class $\phi \in H^*(\P^{N-1})$, let $\phi_{(m)}$ denote the class given by $\phi$ in the twisted sector $H^*(X_{e^{2\pi\i m}}) \cong H^*(\P^{N-1})$.  Let $H \in H^*(\P^{N-1})$ be the hyperplane class, and let
\[F_m := F_{e^{2\pi i m}} = \{i \; | \; mw_i \in \Z\}.\]
On the moduli space $Z^{\epsilon}_{0,(m,-m,1/d),0}$, one calculates that
\[h^1\left(L^{\otimes w_i} \otimes \O\left(-\sum_{k=1}^n [q_k]\right)\right) = \begin{cases} 1 & \text{if } i \in F_m\\ 0 & \text{otherwise,} \end{cases}\]
and the pairing is given by
\[\left( H^j_{(m)}, H^{N-1-|F_m|-j}_{(-m)}\right) = \int_{\P(\vec{d})} H^{N-1-|F_m|} \; e\left(\bigoplus_{i  \in F_m} \O_{\P(\vec{d})}(-w_i)\right) = \frac{1}{d}\prod_{i \in F_m} \left(-\frac{w_i}{d}\right).\]
for any $0 \leq j \leq N-1-|F_m|$, where all other pairings equal zero.

This pairing is degenerate, since whenever $j > N-1 - |F_m|$, the class $H^j_{(m)}$ pairs to zero with every element of $\widetilde{\mathcal{H}}$.  However, it becomes nondegenerate when restricted to the subspace
\[\widetilde{\H}^{\text{ct}} \subseteq \widetilde{\H}\]
generated by $H^j_{(m)}$ with $H^j_{(m)} e\left(\bigoplus_{i \in F_m} \O_{\P^{N-1}}(-w_i)\right) \neq 0$.  This is sufficient for our purposes, because invariants with insertions in the complementary subspace to $\widetilde{\mathcal{H}}^{\text{ct}}$ all vanish, as the next lemma shows.

\begin{lemma}\label{lem:vanishing}
Let $\mathcal{V} = T_{\overline{\mathcal{I}}X/\overline{\mathcal{I}}Z}$ be the relative tangent bundle, whose restriction to $X_{e^{2\pi i m}}$ is $\bigoplus_{i \in F_m} \O_{\P^{N-1}}(-w_i)$.  Let $\phi \in \widetilde{\H}$ be such that
\[\phi \cdot e(\mathcal{V}) = 0.\]
Then
\[\ev_k^*(\phi) \cap [Z^{\epsilon}_{0,n,\beta}]^{\vir} = 0\]
for all $k =1, \ldots, n$.
\begin{proof}
Without loss of generality, let $k=1$.  The lemma will follow if we can prove that
\begin{equation}
\label{eq:goal}
[Z^{\epsilon}_{0,n,\beta}]^{\vir}= \ev_1^*(e(\mathcal{V})) \cdot e\left(\bigoplus_{i=1}^M R^1\pi_*\left(\mathcal{L}^{\otimes w_i} \otimes \O\left(-\sum_{k=2}^n \Delta_k\right)\right)\right) \cap [Z^{\epsilon}_{0,n,\beta}],
\end{equation}
since this will imply that
\[\ev_1^*(\phi) \cap [Z^{\epsilon}_{0,n,\beta}]^{\vir} = \ev_1^*(\phi \cdot e(\mathcal{V})) \cdot e\left(\bigoplus_{i=1}^M R^1\pi_*\left(\mathcal{L}^{\otimes w_i} \otimes \O\left(-\sum_{k=2}^n \Delta_k\right)\right)\right) \cap [Z^{\epsilon}_{0,n,\beta}]\]
and the first factor on the right-hand side is zero by assumption.

In order to ensure that \eqref{eq:goal} makes sense, we must first verify that
\begin{equation}
\label{eq:R0}
R^0\pi_*\left(\mathcal{L}^{\otimes w_i} \otimes \O\left(-\sum_{k=2}^n \Delta_k\right)\right) = 0
\end{equation}
for each $i$, so that the expression inside the Euler class is indeed a bundle.  To check \eqref{eq:R0}, we calculate the degree of $|L^{\otimes w_i} \otimes \O\left(-\sum_{k=2}^n[q_k]\right)|$.  If we denote
\[a^i_k := \begin{cases} 1 & \text{mult}_{q_k}(L^{\otimes w_i}) = 0\\ 0 & \text{mult}_{q_k}(L^{\otimes w_i}) \neq 0,\end{cases}\]
then
\begin{align*}
\deg\left|L^{\otimes w_i} \otimes \O\left(-\sum_{k=2}^n[q_k]\right)\right| &= \deg\left|L^{\otimes w_i} \otimes \O\left(-\sum_{k=2}^na^i_k[\overline{q}_k]\right)\right|\\
&=\deg\left(L^{\otimes w_i}\right) - \sum_{k=1}^n \text{mult}_{q_k}(L^{\otimes w_i}) - \sum_{k=2}^n a^i_k\\
&=\deg\left(L^{\otimes w_i}\right) - \sum_{k=2}^n \left(\text{mult}_{q_k}(L^{\otimes w_i}) + a^i_k\right) - \text{mult}_{q_1}(L^{\otimes w_i}).
\end{align*}
The fact that $w_i | d$ implies that the smallest possible nonzero value for $\mult_{q_k}(L^{\otimes w_i})$ is $\frac{w_i}{d}$, and hence the above is less than or equal to
\[\deg\left(L^{\otimes w_i}\right) - \frac{w_i}{d}(n-1)= \frac{w_i}{d}(n-2-\beta) - \frac{w_i}{d}(n-1) < 0.\]
On an irreducible curve, this is enough to conclude that $L^{\otimes w_i} \otimes \O\left(-\sum_{k=2}^n [q_k]\right)$ has no nonzero global sections.  For reducible curves, one applies an induction on components starting from a tail not containing $q_1$ to deduce, component-by-component, that any global section again must vanish.  This establishes \eqref{eq:R0}.

Now, to prove \eqref{eq:goal}, we use the exact sequence
\[0 \rightarrow \mathcal{L}^{\otimes w_i} \left(-\sum_{k=1}^n\Delta_k\right) \rightarrow \mathcal{L}^{\otimes w_i} \left(-\sum_{k=2}^n\Delta_k\right) \rightarrow \mathcal{L}^{\otimes w_i} \left(-\sum_{k=2}^n\Delta_k\right)\bigg|_{\Delta_1} \rightarrow 0.\]
The associated long exact sequence, together with \eqref{eq:R0}, yields
\[
0 \rightarrow R^0\pi_*\left(\mathcal{L}^{\otimes w_i}\bigg|_{\Delta_1}\right) \rightarrow R^1\pi_*\left(\mathcal{L}^{\otimes w_i} \left(-\sum_{k=1}^n\Delta_k\right)\right) \rightarrow R^1\pi_*\left(\mathcal{L}^{\otimes w_i} \left(-\sum_{k=2}^n\Delta_k\right)\right) \rightarrow 0.
\]
The first term is zero when $\text{mult}_{q_1}(L^{\otimes w_i}) = 0$---or, in other words, when $i \notin F_{m_1}$---since any section of an orbifold bundle vanishes at a point with nonzero multiplicity.  When $i \in F_{m_1}$, this term equals  $\ev_1^*(\O_X(-w_i))$.  Thus, summing over $i$ and taking Euler classes produces exactly \eqref{eq:goal}.
\end{proof}
\end{lemma}

\begin{remark}
\label{rem:ct}
Lemma~\ref{lem:vanishing} verifies Conjecture 2.12 of \cite{CJRGLSM} in genus zero.  In keeping with the language of that paper, the notation ``ct" is chosen to reflect the fact that $\widetilde{\H}^{\ct}$ is isomorphic to the compact-type state space described in \cite[Section 2.1]{CJRGLSM}.
\end{remark}

\subsection{Small $J$-function and the wall-crossing formula}
\label{sec:J}

The $J$-function for $\epsilon$-stable quasimap theory was defined by Ciocan-Fontanine and Kim in \cite{CFKZero}, and was generalized to the spin setting by the second author and Ruan in \cite{RR}.  To define it, we let $\mathcal{G}Z^{\epsilon}_{0,1,\beta}$ be the ``graph space" parameterizing the same data as $Z^{\epsilon}_{0,1,\beta}$ together with a parameterization of one component $C_0 \subseteq C$ on which the ampleness condition of Definition \eqref{def:qmap} is not required.  By \eqref{eq:mults}, the multiplicity of $L$ at the single marked point $q_1$ must be
\[m_1 = \left\langle \frac{-\beta-1}{d}\right\rangle.\]

There is an action of $\C^*$ on $\mathcal{G}Z^{\epsilon}_{0,1,\beta}$ by scaling the parameterized component. More specifically, let $[x_0,x_1]$ be the homogeneous coordinates on $C_0$ and let $\C^*$ act by $t\cdot[x_0,x_1]:=[tx_0,x_1]$. Let $F^{\epsilon}_{\beta} \subseteq \mathcal{G}Z^{\epsilon}_{0,1,\beta}$ be the fixed locus on which $q_1 = \infty=[0:1] \in C_0$ and all of the degree lies over $0=[1:0]$.  When $\beta > 1/\epsilon$, we have
\[F^{\epsilon}_{\beta} \cong Z^{\epsilon}_{0,1,\beta},\]
while when $\beta \leq 1/\epsilon$, the moduli space $Z^{\epsilon}_{0,1,\beta}$ is empty and instead, we have
\[F^{\epsilon}_{\beta} \cong Z,\]
corresponding to quasimaps whose degree is entirely concentrated in a single basepoint at $0 \in C_0$.  In either case, there is an evaluation map
\[\ev_{\bullet}: F^{\epsilon}_{\beta} \rightarrow \overline{\mathcal{I}}Z\]
defined by evaluation at the single marked point $\infty \in C_0$.

\begin{definition}
Let $z$ denote the equivariant parameter for the action of $\C^*$ on $\mathcal{G}Z^{\epsilon}_{0,1,\beta}$.  The {\it small $\epsilon$-stable $J$-function} is
\[J^{\epsilon}(q,z) := -z^2 \sum_{\beta \geq 0 \atop \phi} q^{\beta}\left(\int_{[F^{\epsilon}_{\beta}]^{\vir}}\frac{\ev_\bullet^*(\phi)}{e_{\C^*}(N^{\vir}_{F^{\epsilon}_{\beta}/\mathcal{G}Z^{\epsilon}_{0,1,\beta}})}\right)\phi^\vee \in \widetilde{\mathcal{H}}^{\text{ct}}_{\mathbb{C}^*}[[q]][z,z^{-1}],\]
where $\phi$ runs over any basis of $\widetilde\H^{\text{ct}}$ and $(-)^\vee$ denotes dual under the pairing introduced in Section \ref{sec:pairing}.
\end{definition}

More explicitly, since for $\beta > 1/\epsilon$ we have $[F^{\epsilon}_{\beta}]^{\vir} = [Z^{\epsilon}_{0,1,\beta}]^{\vir}$ and
\[e_{\C^*}(N^{\vir}_{F^{\epsilon}_{\beta}/\mathcal{G}Z^{\epsilon}_{0,1,\beta}}) = -z^2(z-\psi_1),\]
the contribution of such $\beta$ to the $J$-function is 
\begin{equation}
\label{eq:Jstable}
-z^2\sum_{\phi}\left(\int_{[F^{\epsilon}_{\beta}]^{\vir}}\frac{ \ev_\bullet^*(\phi)}{e_{\C^*}(N^{\vir}_{F^{\epsilon}_{\beta}/\mathcal{G}Z^{\epsilon}_{0,1,\beta}})}\right)\phi^\vee = \sum_{l=0}^{N-1-|F_{(\beta+1)/d}|}\left\langle \frac{ H^l_{\left(-\frac{\beta+1}{d}\right)}}{z-\psi_1} \right\rangle_{0,1,\beta}^\epsilon \left(H^l_{\left(-\frac{\beta+1}{d}\right)}\right)^\vee.
\end{equation}
The denominator of \eqref{eq:Jstable} should be understood as a geometric series in $\psi_1$. For $\beta \leq 1/\epsilon$, the contribution to the $J$-function can be calculated explicitly:
\begin{equation}
\label{eq:Junstable}
-z^2\sum_{\phi}\left(\int_{[F^{\epsilon}_{\beta}]^{\vir}}\frac{ \ev_\bullet^*(\phi)}{e_{\C^*}(N^{\vir}_{F^{\epsilon}_{\beta}/\mathcal{G}Z^{\epsilon}_{0,1,\beta}})}\right)\phi^\vee = z\frac{\displaystyle\prod_{i=1}^M\prod_{\substack{0 < b < \frac{w_i}{d}(\beta+1)\\ \langle b \rangle = \langle \frac{w_i}{d}(\beta+1) \rangle}} \left(-bz - \frac{w_i}{d}H_{\left(\frac{\beta+1}{d}\right)}\right)}{\displaystyle\prod_{j=1}^N\prod_{\substack{0 < b \leq \beta \\ \langle b \rangle = 0}} \left(bz + H_{\left(\frac{\beta+1}{d}\right)}\right)}\in \widetilde\H^{\text{ct}}[z,z^{-1}].
\end{equation}
Again, the denominator of \eqref{eq:Junstable} should be understood as a geometric series in $H_{\left(\frac{\beta+1}{d}\right)}$ and the series should be truncated to lie in $\widetilde\H^{\text{ct}}\subset\widetilde\H$. Terms of $J^{\epsilon}$ with $\beta \leq 1/\epsilon$ are referred to as {\it unstable terms}.
Denote by
\[[J^{\epsilon}]_+(q,z) \in \widetilde{\mathcal{H}}^{\text{ct}}[[q,z]]\]
the part of the $J$-function with non-negative powers of $z$, which has contributions only from the unstable terms. The coefficients in the change of variables in Theorem \ref{thm:2}, denoted by $\mu^{\epsilon}_{\beta}(z)$, are defined by
\[
\sum_{\beta \geq 0} q^{\beta} \mu^{\epsilon}_{\beta}(z) = [J^{\epsilon}]_+(q,z)-\one z,
\]
and they are determined explicitly by \eqref{eq:Junstable}.

We also require a generalization of the small $J$-function in which descendent insertions are allowed.  This is the {\it big $\epsilon$-stable $J$-function}, defined for $\mathbf{t} = \mathbf{t}(z) \in \widetilde\H^{\text{ct}}[[z]]$ by
\begin{equation}
\label{eq:bigJ}
J^{\epsilon}(q,\mathbf{t},z) := -z^2 \sum_{\substack{\beta \geq 0\\ n \geq 0 \\ \phi}} \frac{q^{\beta}}{n!} \left(\int_{[F^{\epsilon}_{\beta}]^{\vir}} \prod_{k=1}^n \ev_k^*(\mathbf{t}(\psi_k)) \cap \frac{\ev_{\bullet}^*(\phi)}{e_{\C^*}(N^{\vir}_{F^{\epsilon}_{n,\beta}/\mathcal{G}Z^{\epsilon}_{0,n+1,\beta}})}\right)\phi^\vee,
\end{equation}
where $\phi$ again runs over a basis for $\widetilde{\H}^{\text{ct}}$ and $\mathcal{G}Z^{\epsilon}_{0,n+1,\beta}$ is the $(n+1)$-pointed analogue of the above-defined graph space.  Inside this graph space, $F^{\epsilon}_{n,\beta}$ is the fixed locus where all but the last marked point and all of the degree are concentrated over $0 \in C_0$ while the last marked point lies at $\infty \in C_0$.  The small $J$-function is recovered from the big $J$-function by setting $\mathbf{t} = 0$.

\section{Localization framework}
\label{sec:localization}

There is an action of the torus $\mathbf{T} = (\C^*)^N$ on $Z$ by diagonal multiplication on the $p$-coordinates.  This induces an action on $Z^{\epsilon}_{0,n,\beta}$ by post-composition, or in other words, by scaling the sections $\vec{p}$.  The action naturally lifts to the bundle $R^1\pi_*\left(\mathcal{L}^{\otimes w_i} \otimes \O\left(-\sum_{k=1}^n \Delta_k\right)\right)$ for each $i$, and we let
\[[Z^{\epsilon}_{0,n,\beta}]^{\vir}_{\mathbf{T}} = e_{\mathbf{T}}\left(\bigoplus_{i=1}^M R^1\pi_*\left(\mathcal{L}^{\otimes w_i} \otimes \O\left(-\sum_{k=1}^n \Delta_k\right)\right)\right) \cap [Z^{\epsilon}_{0,n,\beta}]\]
be the $\mathbf{T}$-equivariant virtual cycle.

Let $\alpha_1, \ldots, \alpha_N$ denote the equivariant parameters for the $\mathbf{T}$-action.  Then, by the localization isomorphism, we have
\begin{equation}
\label{eq:Hj}
\widetilde{\mathcal{H}}_{\mathbf{T}} := H_{\mathrm{CR},\mathbf{T}}^*(X) \otimes \C(\alpha_1, \ldots, \alpha_N) = \bigoplus_{j=1}^N H^*_{\text{CR},\mathbf{T}}(P_j) \otimes \C(\alpha_1, \ldots, \alpha_N)=:\bigoplus_{j=1}^N\widetilde{\mathcal{H}}_j,
\end{equation}
where $P_j$ is the unique $\mathbf{T}$-fixed point of $X$ with $p_j \neq 0$; that is, $P_j = [0: \cdots :1: \cdots : 0] \in \P(d,\ldots, d) \subseteq X$, with a $1$ in the $j$th position.  For each $m \in \{0,\frac{1}{d}, \ldots, \frac{d-1}{d}\}$, we denote by $\mathbf{1}^j_{(m)}$ the fundamental class on the twisted sector of $\widetilde\H_j$ indexed by $m$.  The classes $\{\mathbf{1}^j_{(m)}\}$ form a basis of $\widetilde{\mathcal{H}}_{\mathbf{T}}$ as a $\C(\alpha_1, \ldots, \alpha_N)$-module, which we call the {\it fixed-point basis}. 

The pairing on $\widetilde\H$ lifts to a pairing on $\widetilde\H_{\mathbf{T}}$, defined by
\[
\left(H_{(m_1)}^{j_1},H_{(m_2)}^{j_2}\right)=\delta_{m_1+m_2\in\Z}\int_{\mathbb{P}(\vec d)}H^{j_1+j_2}e_{\mathbf{T}}\left(\bigoplus_{i\in F_{m_1}}\O_{\mathbb{P}(\vec d)}(-w_i)  \right),
\]
and this equivariant pairing is non-degenerate because the equivariant Euler class is invertible. In the fixed-point basis, the equivariant pairing on $\widetilde{\mathcal{H}}_{\mathbf{T}}$ is
\begin{equation}\label{eq:fixedpairing}
\left(\mathbf{1}^{j_1}_{(m_1)},\mathbf{1}^{j_2}_{(m_2)}\right)=\delta_{m_1+m_2=0\text{ mod } d}\; \frac{\prod_{i  \;| \langle w_i m_1 \rangle = 0} \left(-\frac{w_i}{d}\alpha_{j}\right)} {d\prod_{j' \neq j} \left( \alpha_{j} - \alpha_{j'}\right)}:=\eta_{(m_1)}^{j_1}.
\end{equation}

 \subsection{Fixed loci}
 
 The fixed loci of the $\mathbf{T}$-action on $Z^{\epsilon}_{0,n,\beta}$ are indexed by decorated trees.  For a tree $\Gamma$, let $V(\Gamma)$, $E(\Gamma)$, and $F(\Gamma)$ denote the sets of vertices, edges, and flags, respectively.  Localization trees are decorated as follows:
 \begin{itemize}
 \item Each vertex $v$ is decorated by an index $j_v \in \{1, \ldots, N\}$ and a degree $\beta_v \in \mathbb{N}$.
 \item Each edge $e$ is decorated by a degree $\beta_e \in \mathbb{N}_{>0}$.
 \item Each flag $(v,e)$ is decorated by a multiplicity $m_{(v,e)} \in\{0,\frac{1}{d},\dots,\frac{d-1}{d}\}$.
\end{itemize}
In addition, $\Gamma$ is equipped with a map
\[s: \{1,\ldots,n\} \rightarrow V(\Gamma)\]
assigning marked points to the vertices.  Let $E_v$ be the set of edges incident to a vertex $v$, and define the {\it valence} of $v$ as
\[\text{val}(v):= |E_v| + |s^{-1}(v)|.\]

Given a tree $\Gamma$ with the above decorations, the fixed locus $F_{\Gamma} \subseteq Z^{\epsilon}_{0,n,\beta}$ indexed by $\Gamma$ parameterizes Landau--Ginzburg quasimaps as follows:
 \begin{itemize}
   \item Each vertex $v \in V(\Gamma)$ corresponds to a connected component $C_v \subseteq C$ over which $p_j = 0$ for $j \neq j_v$, and $\beta_v$ is the degree of the restriction of $L^{\otimes -d} \otimes \omega_{\log}$ to $C_v$.  If $Z^{\epsilon}_{0,\text{val}(v),\beta_v} \neq \emptyset$, then $C_v$ is a sub-curve and we say $v$ is {\it stable}.  If $Z^{\epsilon}_{0,\text{val}(v),\beta_v} = \emptyset$, then we say $v$ is {\it unstable}, and $C_v$ is the single point $q_v$; if $\beta_v > 0$, this point is a basepoint on $C_e$ of order $\beta_v$.
 \item Each edge $e \in E(\Gamma)$ with adjacent vertices $v$ and $v'$ corresponds to an orbifold projective line $C_e$ over which $p_j = 0$ for $j \neq j_v,j_{v'}$.  The section $p_{j_v}$ vanishes only at a single point $q_{v'}$, while the section $p_{j_{v'}}$ vanishes only at a single point $q_v$, and $\beta_e$ is the part of $\deg(L^{\otimes -d} \otimes \omega_{\log}|_{C_e})$ not coming from basepoints.  In other words,
\[\beta_e = \begin{cases} \deg(L^{\otimes -d} \otimes \omega_{\log}|_{C_e}) & v,v' \text{ stable }\\  \deg(L^{\otimes -d} \otimes \omega_{\log}|_{C_e})- \beta_v & v  \text{ unstable}\\ \deg(L^{\otimes -d} \otimes \omega_{\log}|_{C_e}) - \beta_{v'} & v'  \text{ unstable.} \end{cases}\]
 \item The set $s^{-1}(v) \subseteq \{1,\ldots, n\}$ indexes the marked points supported on $C_v$.
 \item For flags $(v,e) \in F(\Gamma)$:
 \begin{enumerate}[(i)]
 \item  If $v$ is stable, then the flag $(v,e)$ corresponds to a node attaching $C_v$ to $C_e$ and $m_{(v,e)}$ is the multiplicity of $L$ on the vertex branch of the node.
 \item If $v$ is unstable of valence two, then either (1) $|E_v| =2$ and $|s^{-1}(v)| = 0$, in which case the flag $(v,e)$ corresponds a node attaching $C_e$ to the component $C_{e'}$ associated to the other edge $e'$ incident to $v$, and $m_{(v,e)}$ is the multiplicity of $L$ at $q_v \in C_{e'}$; or (2) $|E_v| = 1$ and $|s^{-1}(v)| = 1$, in which case the flag $(v,e)$ corresponds to a marked point point at $q_v \in C_e$ and $-m_{(v,e)}$ is the multiplicity of $L$ at $q_v$.
 \item If $v$ is unstable of valence one, then the flag $(v,e)$ corresponds to the unmarked point $q_v \in C_e$ and $m_{(v,e)} =  -\frac{\beta_v + 1}{d} \mod \Z$.
 \end{enumerate}
 \end{itemize}

The localization formula expresses $[Z^{\epsilon}_{0,n,\beta}]^{\vir}_{\mathbf{T}}$ in terms of contributions from each localization tree $\Gamma$.  To make this more explicit, define a moduli space $\M^{1/d,\epsilon}_{0,n,\beta}$ that parameterizes tuples
\[(C; q_1, \ldots, q_n; L;D; \varphi),\]
where $(C; q_1, \ldots, q_n; L)$ are as usual, $D$ is a divisor on $C$ of degree $\beta$, and $\varphi$ is an isomorphism
\[\varphi: L^{\otimes -d} \otimes \omega_{\log} \xrightarrow{\sim} \O(D).\]
We assume the usual representability and stability conditions, as well as the nondegeneracy condition that if
\[D = \sum_k b_k[y_k]\]
for distinct points $y_k \in C$, then the points $y_k$ are disjoint from the marked points and nodes and $b_k \leq 1/\epsilon$ for each $k$.  (This is a special case of the moduli spaces of ``weighted spin curves" studied in \cite{RR}.)

For each element of the fixed locus associated to a localization tree $\Gamma$ and each stable vertex $v$, one obtains an element of $\M^{1/d,\epsilon}_{0,\text{val}(v),\beta_v}$ by taking $D$ to be the zero locus of $p_{j_v}$.  Thus, if we let
\[
F_{\Gamma} = \prod_{v \text{ stable}} \M^{1/d,\epsilon}_{0,\text{val}(v),\beta_v},
\]
then there is a canonical family of $\mathbf{T}$-fixed elements of $Z^{\epsilon}_{0,n,\beta}$ over $F_{\Gamma}$, which yields a morphism
 \[
 \iota_{\Gamma}: F_{\Gamma} \rightarrow Z^{\epsilon}_{0,n,\beta}
 \]
 that is \'etale onto the fixed locus corresponding to $\Gamma$.  The localization formula then yields
 \begin{equation}
 \label{eq:locstart}
 [Z^{\epsilon}_{0,n,\beta}]^{\vir}_{\mathbf{T}} = \sum_{\Gamma} \frac{1}{|\text{Aut}(F_\Gamma)|} \iota_{\Gamma *} \left( \frac{[F_{\Gamma}]^{\vir}_{\mathbf{T}}}{e_{\mathbf{T}}(N_{\Gamma}^{\vir})}\right),
 \end{equation}
 where $\text{Aut}(F_\Gamma)$ is the group of automorphisms of a generic element of the fixed locus indexed by $\Gamma$. In the next subsection, we calculate $|\text{Aut}(F_\Gamma)|$ and $e^{-1}_{\mathbf{T}}(N_{\Gamma}^{\vir})$ explicitly.

\subsection{Localization contributions}
\label{subsec:loccontr}

For each localization tree $\Gamma$, the localization contribution can be divided into vertex, edge, and flag contributions, following the standard argument that has appeared in \cite{GP} and in many other contexts.  To summarize, one applies the normalization exact sequence to the relative obstruction theory on $Z^{\epsilon}_{0,n,\beta}$ to express it in terms of contributions on vertex components, edge components, and nodes.  This accounts for all but the automorphisms and deformations of $(C,q_1,\dots,q_n,L)$.  The latter are comprised of deformations of the vertex components and their bundle $L$ (included in the vertex contributions below), automorphisms/deformations of the edge components and their $L$ (included in the edge contributions), and deformations smoothing the nodes (included in the flag contributions).

In the end, we write 
\begin{align}
\label{eq:locfinal}
&\ev_1^*(\phi_1) \psi_1^{a_1} \cdots \ev_n^*(\phi_n)\psi_n^{a_n} \cap [Z^{\epsilon}_{0,n,\beta}]^{\vir}_{\mathbf{T}} =\\
 \nonumber &\sum_{\Gamma} \frac{1}{|\text{Aut}(\Gamma)|} \iota_{\Gamma*}\left(\prod_{v \in V(\Gamma)} \text{Contr}_{\Gamma}(v)\prod_{(v,e)\in F(\Gamma)}\text{Contr}_\Gamma(v,e)\right)\prod_{e \in E(\Gamma)} \text{Contr}_{\Gamma}(e),
\end{align}
where $\text{Contr}_\Gamma(v)$, $\text{Contr}_\Gamma(e)$, and $\text{Contr}_\Gamma(v,e)$ are described below. 

\subsubsection{Stable vertex contributions} First, let $v$ be a stable vertex of $\Gamma$.  The deformations of the marked curve $C_v$ and the line bundle $L|_{C_v}$ are $\mathbf{T}$-fixed, so they contribute to the virtual fundamental cycle:
\begin{equation}\label{eq:restrictvir}
e_{\mathbf{T}}^{-1}\left(R\pi_*\left(\bigoplus_{i=1}^M \left(\mathcal{L}^{\otimes w_i} \otimes \O\left(-\displaystyle\sum_{k \in s^{-1}(v)} \hspace{-0.25cm}\Delta_k\right)\right)\otimes \C_{-\frac{w_i\alpha_{j_v}}{d}} \right)\right),
\end{equation}
where $\mathbb{C}_{\alpha}$ is the topologically trivial line bundle with equivariant first Chern class $\alpha$. There is an asymmetry in \eqref{eq:restrictvir} in that the universal bundles are only twisted down by the marked points, and not at the pre-images of nodes. To correct the asymmetry, we note that \eqref{eq:restrictvir} is equal to 
\begin{equation}\label{eq:restrictvir2}
 \prod_{i  \;| \langle w_i m_{(v,e)}\rangle = 0} \left(-\frac{w_i\alpha_{j_v}}{d}\right)^{-1}e_{\mathbf{T}}^{-1}\left(R\pi_*\left(\bigoplus_{i=1}^M \left(\mathcal{L}^{\otimes w_i} \otimes \O\left(-\displaystyle\sum_{k \in s^{-1}(v)\cup E_v} \hspace{-0.25cm}\Delta_k\right)\right)\otimes \C_{-\frac{w_i\alpha_{j_v}}{d}} \right)\right),
 \end{equation}
where $E_v$ is the set of edges incident to $v$. This equality can be readily checked by using the long exact sequence in cohomology associated to the short exact sequence
\[
0\rightarrow L(-q_e)\rightarrow L \rightarrow L|_{q_e}\rightarrow 0,
\]
where $q_e$ is the pre-image of any node.

The deformations of the section $p_j$ are moving for $j\neq j_v$, and their contribution to the inverse Euler class of the virtual normal bundle is
\[e_{\mathbf{T}}^{-1}\left(\bigoplus_{j \neq j_v}R\pi_*\left(\mathcal{L}^{\otimes -d} \otimes \omega_{\log}\right) \otimes \C_{\alpha_{j_v} - \alpha_j}\right).\]

In addition, for each edge $e$ adjacent to $v$, there is a contribution from deformations smoothing the node at which $C_e$ meets $C_v$, and a gluing factor of $d$, which yields
\begin{equation*}
\prod_{e \in E_v} \frac{d}{\frac{\alpha_{j_{v}} - \alpha_{j_v'}}{\beta_e} - \psi_{(v,e)}},
\end{equation*}
where each edge $e\in E_v$ joins $v$ to another vertex $v'$, and $\psi_{(v,e)}$ is the cotangent line class to the coarse curve at the vertex branch of the node where $C_v$ and $C_e$ meet. The factor of $d$ will be absorbed into the flag term.

Motivated by these computations, for each stable vertex, we define
\begin{equation}
\label{eq:stablevertex}
\text{Contr}_{\Gamma}(v) = \prod_{k \in s^{-1}(v)} \ev_k^*\left(\phi_k|_{P_{j_v}}\right)\psi_k^{a_k} \prod_{e \in E_v} \frac{\ev_e^*\left(\one^{j_v}_{(m_{(v,e)})}\right)}{\frac{\alpha_{j_{v}} - \alpha_{j_{v'}}}{\beta_e} - \psi_{(v,e)}} \cap [\M^{1/d, \epsilon}_{0,\text{val}(v),\beta_v}]^{\vir, j_v}_{\mathbf{T}},
\end{equation}
where $\ev_e$ is the evaluation map at the half-node corresponding to the edge $e$ and
\begin{align}
\label{eq:twistedvcycle}
&[\M^{1/d, \epsilon}_{0,\text{val}(v),\beta_v}]^{\vir, j_v}_{\mathbf{T}} := [\M^{1/d, \epsilon}_{0,\text{val}(v),\beta_v}] \cap \\
\nonumber&e_{\mathbf{T}}^{-1}\left(R\pi_*\left(\bigoplus_{i=1}^M \left(\mathcal{L}^{\otimes w_i} \left(-\hspace{-0.25cm}\displaystyle\sum_{k \in s^{-1}(v)\cup E_v} \hspace{-0.25cm}\Delta_k\right)\right)\otimes \C_{-\frac{w_i\alpha_{j_v}}{d}} \oplus \bigoplus_{j \neq j_v}\left(\mathcal{L}^{\otimes -d} \otimes \omega_{\log}\right)\otimes \C_{\alpha_{j_v} - \alpha_j}\right)\right).
\end{align}

\subsubsection{Edge contributions}

Let $e$ be an edge with adjacent stable vertices $v$ and $v'$.  The marked curve $C_e$ and its line bundle $L|_{C_e}$ have no fixed deformations.  We calculate, together, the contribution from the moving deformations of $\vec{p}$ to the inverse Euler class of the virtual normal bundle and the edge contribution to the virtual fundamental cycle.  This combined contribution is
\[\frac{e_{\mathbf{T}}(\bigoplus_{i=1}^M H^1(C_e,L^{\otimes w_i}))}{e_{\mathbf{T}}(\bigoplus_{j=1}^N H^0(C_e,L^{\otimes -d})^{\text{mov}})},\]
where the superscript ``mov" denotes the moving part with respect to the $\mathbf{T}$-action.  (Here, we use that $\omega_{C_e,\log} \cong \O_{C_e}$.)  The above can be calculated explicitly as:
\begin{equation}
\label{eq:stable_edge}
 \frac{\prod_{i=1}^M \prod_{\substack{0 < b < \frac{\beta_ew_i}{d}\\ \langle b \rangle = \langle w_im_{(v,e)}\rangle}} \left(\frac{b}{\beta_e}(\alpha_{j_{v}} - \alpha_{j_{v'}}) - \frac{w_i}{d}\alpha_{j_{v}}\right)}{\prod_{j=1}^N \prod'_{\substack{0 \leq b \leq \beta_e\\ \langle b \rangle = 0}} \left(\frac{b}{\beta_e}(\alpha_{j_{v'}} - \alpha_{j_{v}}) + \alpha_{j_v} - \alpha_{j}\right)},
\end{equation}
where $\prod'$ in the denominator denotes the product over all nonzero factors. 

Notice that $L|_{C_e}$ is not quite what one would expect on an edge, because it is not twisted down at pre-images of nodes, which we think of as marked points on $C_e$. Twisting down at nodes results in changing the strict inequalities in the numerator of \eqref{eq:stable_edge} to non-strict inequalities. Combining this with the automorphisms of the edge, we define 
\begin{equation}
\label{eq:contr_e}
\text{Contr}_{\Gamma}(e) =  \frac{1}{d\beta_e}\cdot \frac{\prod_{i=1}^M \prod_{\substack{0 \leq b \leq \frac{\beta_ew_i}{d}\\ \langle b \rangle = \langle w_im_{(v,e)}\rangle}} \left(\frac{b}{\beta_e}(\alpha_{j_{v}} - \alpha_{j_{v'}}) - \frac{w_i}{d}\alpha_{j_{v}}\right)}{\prod_{j=1}^N \prod'_{\substack{0 \leq b \leq \beta_e\\ \langle b \rangle = 0}} \left(\frac{b}{\beta_e}(\alpha_{j_v'} - \alpha_{j_{v}}) + \alpha_{j_v} - \alpha_{j}\right)}
\end{equation}
as the total (stable) edge contribution.

\subsubsection{Flag contributions}

For each flag $(v,e)$ at $v$, there is a contribution to the normalization exact sequence from the corresponding node.  This equals
\begin{equation}
\label{eq:node}
e_{\mathbf{T}}\left(N_{j_v}^{m_{(v,e)}}\right) = \prod_{i  \;| \langle w_i m_{(v,e)}\rangle = 0} \left(-\frac{w_i}{d}\alpha_{j_v}\right) \prod_{j \neq j_v} \left( \alpha_{j_v} - \alpha_{j}\right),
\end{equation}
where $N_{j_v}^{m_{(v,e)}}$ is the normal bundle of the $\mathbf{T}$-fixed point $P_{j_v}$ in $X_{e^{2\pi\i m_{(v,e)}}}$. We multiply this contribution by $d$ (from the gluing term at nodes), and we multiply it by $\prod_{i  \;| \langle w_i m_{(v,e)}\rangle = 0} \left(-\frac{w_i}{d}\alpha_{j_v}\right)^{-2}$ to compensate for the factors arising from twisting down at both pre-images of the nodes. Altogether, we define
\[
\text{Contr}_{(v,e)}:=\left(\eta_{(m_{(v,e)})}^{j_v}\right)^{-1},
\]
where $\eta$ is the coefficient of the pairing \eqref{eq:fixedpairing}.

\subsubsection{Unstable vertex contributions}

We now describe the conventions for the unstable vertices, which are defined to ensure that \eqref{eq:locfinal} holds with edge and flag contributions defined as above.  For $v$ such that $|E_v|=2$ and $|s^{-1}(v)| = 0$, then by smoothing the node and compensating for the two flag terms, we define
\[
\text{Contr}_{\Gamma}(v)=\frac{\eta_{(m_{(v,e)})}^{j_v}}{\frac{\alpha_{j_{v}} - \alpha_{j_{v_1}}}{\beta_{e_1}} + \frac{\alpha_{j_{v}} - \alpha_{j_{v_2}}}{\beta_{e_2}}}.
\]
For $v$ such that $|E_v|=1$ and $|s^{-1}(v)| =1$, then restricting the insertion to $q_v$ and compensating for the flag term, we define
\begin{equation}
\label{eq:unstablevertex1}
\text{Contr}_{\Gamma}(v) =\phi_{k_v}|_{P_{j_v}}\left(\frac{\alpha_{j_{v'}} - \alpha_{j_{v}}}{\beta_e}\right)^{a_{k_v}} \eta_{(m_{(v,e)})}^{j_v},
\end{equation}
where we write $s^{-1}(v)=\{k_v\}$. Finally, let $e$ be an edge with adjacent vertices $v$ and $v'$ such that $v'$ is stable but $v$ is unstable with $E_v=\{e\}$ and $s^{-1}(v)=\emptyset$. Then a tedious but direct computation shows that
\[
\frac{e_{\mathbf{T}}(\bigoplus_{i=1}^M H^1(C_e,L^{\otimes w_i}))}{e_{\mathbf{T}}(\bigoplus_{j=1}^N H^0(C_e,L^{\otimes -d})^{\text{mov}})}=\frac{\prod_{i=1}^M\prod_{-\frac{(\beta_v+1)w_i}{d}<b<\frac{\beta_ew_i}{d} \atop \langle b \rangle =\langle w_im_{(v,e)}\rangle}\left(\frac{b}{\beta_e}(\alpha_{j_{v}} - \alpha_{j_{v'}}) - \frac{w_i}{d}\alpha_{j_{v}}\right) }{\prod_{j=1}^N \prod'_{\substack{-\beta_v \leq b \leq \beta_e\\ \langle b \rangle = 0}} \left(\frac{b}{\beta_e}(\alpha_{j_v'} - \alpha_{j_{v}}) + \alpha_{j_v} - \alpha_{j}\right)}.
\]
Removing the stable edge contribution, compensating for the flag term, and accounting for the infinitesimal automorphism at $q_v$, this motivates defining the unstable vertex contribution as follows:
\begin{equation}
\label{eq:unstable_vertex}
\text{Contr}_{\Gamma}(v) =\eta_{(m_{(v,e)})}^{j_v}\frac{\alpha_{j_{v}}-\alpha_{j_{v'}}}{\beta_e}\cdot \frac{\prod_{i=1}^M \prod_{\substack{0 < b < \frac{w_i}{d}(\beta_v+1) \\ \langle b \rangle = \langle -w_im_{(v,e)}\rangle}} \left(\frac{b}{\beta_e}(\alpha_{j_{v'}} - \alpha_{j_{v}}) - \frac{w_i}{d} \alpha_{j_v}\right)}{\prod_{j=1}^N \prod'_{\substack{0 < b \leq \beta_v\\ \langle b \rangle = 0 }} \left(\frac{b}{\beta_e}(\alpha_{j_v} - \alpha_{j_{v'}}) + \alpha_{j_v} - \alpha_{j}\right)}.
\end{equation}
It can readily be checked that this convention also make sense for edges with two unstable vertices.

\section{Wall-crossing at vertices}
\label{sec:vertexWC}

Having established the localization set-up, the first step toward the proof of Theorems~\ref{thm:1} and \ref{thm:2} is to prove an analogous statement at each vertex of each localization graph.  Exactly as in Section~\ref{sec:J}, one can define a graph space $\mathcal{G}\M_{0,1,\beta}^{1/d,\epsilon}$ parameterizing the same data as $\M_{0,1,\beta}^{1/d,\epsilon}$ together with a parameterization of one component $C_0 \subseteq C$ on which the ampleness condition of Definition~\ref{def:qmap} is not required.  There is an action of $\C^*$ on $\mathcal{G}\M_{0,1,\beta}^{1/d,\epsilon}$ scaling the parameterized component, and we denote by $V^{\epsilon}_{\beta} \subseteq \mathcal{G}\M_{0,1,\beta}^{1/d,\epsilon}$ the fixed locus on which the single marked point lies at $\infty\in C_0$ and all of the other marked points and the basepoints lie over $0 \in C_0$.

For each $j \in \{1,\ldots, N\}$, there is a twisted, $\mathbf{T}$-equivariant relative obstruction theory on $\mathcal{G}\M_{0,1,\beta}^{1/d,\epsilon}$ given by
\[
\mathbb{E}^{\bullet}_{j} = -R\pi_*\left( \bigoplus_{i=1}^M \left(\mathcal{L}^{\otimes w_i}\left(-\Delta_1\right)\right) \oplus \bigoplus_{j'\neq j}(\mathcal{L}^{\otimes -d} \otimes \omega_{\pi,\log}) \right)^{\vee},
\]
where the $\mathbf{T}$-weights are as in \eqref{eq:twistedvcycle}. This is a vector bundle, so the Poincar\'e dual of its top Chern class defines a twisted, $\mathbf{T}$-equivariant virtual cycle on the graph space.  Restricting to the fixed locus $V^{\epsilon}_{\beta}$, we have $V^{\epsilon}_{\beta} \cong \M^{1/d,\epsilon}_{0,1,\beta}$ when $\M^{1/d,\epsilon}_{0,1,\beta}\neq\emptyset$, and $[V^{\epsilon}_{\beta}]^{\vir,{j}}_{\mathbf{T}}$ agrees in this case with \eqref{eq:twistedvcycle}.  When $\beta \leq 1/\epsilon$, on the other hand, the fixed locus $V^{\epsilon}_{\beta}$ is a single (orbifold) point.

Using the evaluation map
\[
\ev_{\bullet}: V^{\epsilon}_{\beta} \rightarrow \overline{\mathcal{I}}\mathcal{B}\Z_d
\]
at the marked point, we define the twisted, $\mathbf{T}$-equivariant vertex $J$-function by
\begin{equation}
\label{eq:vertexJ}
J^{\epsilon,j}(q,z):= -z^2 \sum_{\beta \geq 0 \atop m\in\{0,\frac{1}{d},\dots,\frac{d-1}{d}\}} q^{\beta} \left(\int_{[V^{\epsilon}_{\beta}]^{\vir, {j}}_{\mathbf{T}}} \frac{\ev_{\bullet}^*(\one_{(m)}^{j}) }{e_{\C^*}(N^{\vir}_{V^{\epsilon}_{\beta}/\mathcal{G}\M^{1/d,\epsilon}_{0,1,\beta}})}\right)(\one_{(m)}^{j})^\vee \in \widetilde\H_{j}[[q]]((z)),
\end{equation}
where $\widetilde{\H}_j$ is defined in \eqref{eq:Hj} and the dual is with respect to the pairing in \eqref{eq:fixedpairing}.

We define the big version $J^{\epsilon,j}(q,\mathbf{t},z)$ for $\mathbf{t} \in \widetilde\H_{j}[z]$ similarly.  In particular, for $\beta \leq 1/\epsilon$, the $q^{\beta}$-coefficient in \eqref{eq:vertexJ} is
\[
z\frac{\prod_{i=1}^M \prod_{\substack{0 < b < \frac{w_i}{d}(\beta+1)\\ \langle b \rangle = \left\langle \frac{w_i}{d}(\beta+1)\right\rangle}} \left(-bz - \frac{w_i}{d}\alpha_{j}\right)}{\prod_{j'=1}^N \prod'_{\substack{0 < b \leq \beta\\ \langle b \rangle = 0}} (bz + \alpha_{j} - \alpha_{j'})}\one_{\left(\frac{\beta+1}{d}\right)}^{j},
\]
while for $\beta>1/\epsilon$, the $q^{\beta}$-coefficient in \eqref{eq:vertexJ} is
\[
\sum_{m\in\{0,\frac{1}{d},\dots,\frac{d-1}{d}\}}\left\langle \frac{\one_{(m)}^{j}}{z-\psi_1} \right\rangle_{0,1,\beta}^{\epsilon,j}(\one_{(m)}^{j})^\vee.
\]

Let $\nu^{\epsilon,j}_{\beta}(z)$ be the $q^{\beta}$-coefficient in $[J^{\epsilon,j}]_+(q,z)-\one^j_{(1/d)}$, where the rational functions are expanded as Laurent series in $z$. With this notation established, we state the vertex wall-crossing theorem.

\begin{theorem}
\label{thm:vertexWC}
For any $n\geq 1$, and for any $j \in \{1,\ldots, N\}$, one has
\[ 
\frac{[\M_{0,n,\beta}^{1/d,\epsilon}]^{\vir,j}_{\mathbf{T}}}{z-\psi_n} =\sum_{\beta_1+ \dots+ \beta_k=\beta} \frac{1}{k!} b_{\bbeta*}\left(\prod_{i = 1}^k \ev_{n + i}^*(\nu_{\beta_i}^{\epsilon,j}(-\psi_{n + i})) \cap\frac{[\M_{0,n+k,0}^{1/d,\infty}]^{\vir,{j}}_{\mathbf{T}}}{z-\psi_n} \right).
\]
When $n=1$ and $\beta\leq 1/\epsilon$, we use the convention that
\begin{equation}\label{eq:convention1}
\frac{[\M_{0,1,\beta}^{1/d,\epsilon}]^{\vir,j}_{\mathbf{T}}}{z-\psi_1}:=[q^\beta]J^{\epsilon,j}(q,z)\in\widetilde\H_j[z],
\end{equation}
and when $n=k=1$ in the right-hand side, we use the convention that
\begin{equation}\label{eq:convention2}
\ev_{2}^*(\mathbf{t}(-\psi_{2})) \cap\frac{[\M_{0,2,0}^{1/d,\infty}]^{\vir,j}_{\mathbf{T}}}{z-\psi_1}:=\mathbf{t}(z)\in\widetilde\H_j[z].
\end{equation}
(Here, for a power series $F(q)$, the notation $[q^{\beta}]F(q)$ refers to the coefficient of $q^{\beta}$.)

\begin{remark}\label{rmk:unstable}
If we further make the convention that $\ev_1^*(\one_{(m)}^j):\widetilde\H_j\rightarrow H_{\mathbf{T}}^*(\text{pt})$ is the map $\phi\mapsto (\one_{(m)}^j,\phi)$, then conventions \eqref{eq:convention1} and \eqref{eq:convention2} imply that all of the vertex contributions in the localization formula, including the unstable ones, can be written uniformly as \eqref{eq:stablevertex}.
\end{remark}

\begin{proof}
Expanding both sides as Laurent series at $z=0$, the only contribution to the regular part comes from the unstable contributions \eqref{eq:convention1} and \eqref{eq:convention2}. Since the regular parts of the theorem are already in agreement by the definition of $\nu_{\beta}^{\epsilon,j}$, it remains to prove that the two sides of the theorem agree in their principal parts. We proceed by lexicographic induction on $(\beta,n)$. For the base case $\beta=0$, both sides are equal by observation. 

Now suppose $\beta>0$, and let us first focus on the left-hand side of the theorem. Consider the graph space $\mathcal{G}\M_{0,n,\beta}^{1/d,\epsilon}$ along with the map $\rho:\mathcal{G}\M_{0,n,\beta}^{1/d,\epsilon}\rightarrow \M_{0,n,\beta}^{1/d,\epsilon}$ that forgets the parametrization and stabilizes. In the case that $n=1$ and $\beta\leq 1/\epsilon$, we make the convention that $\M_{0,1,\beta}^{1/d,\epsilon}=\mathcal{B}\mathbb{Z}_d$ and $\rho$ is ${\ev}_\bullet$ followed by the map that takes any class to its dual under the twisted pairing.

There is a $\mathbf{T}$-equivariant substack $\Theta\subseteq\mathcal{G}\M_{0,n,\beta}^{1/d,\epsilon}$ parametrizing elements of $\mathcal{G}\M_{0,n,\beta}^{1/d,\epsilon}$ where the last marked point lies over $\infty\in C_0$ and at least one of the basepoints lies over $0\in C_0$. Since the virtual class restricted to $\Theta$ is an equivariant class, it follows that $\rho_*[\Theta]^{\vir,j}_{\mathbf{T}}$ is regular at $z=0$. Inverting $z$ and computing $\rho_*[\Theta]^{\vir,j}_{\mathbf{T}}$ by localization, there are three types of fixed loci:
\begin{enumerate}
\item $\Theta_\infty$, where $\infty\in C_0$ is a smooth marked point of $C$, meaning that all of the basepoints and the first $n-1$ marked points lie over $0$,
\item $\Theta_{n_1,\beta_1|n_2,\beta_2}$, where the basepoints and marked points split up over $0$ and $\infty$ in a stable way such that neither $0$ nor $\infty$ in $C_0$ are smooth points of $C$, and
\item $\Theta_{\beta_1}$, where $0\in C_0$ is a smooth basepoint of order $\beta_1\leq1/\epsilon$.
\end{enumerate}
The three types of fixed loci contribute to give
\begin{align*}
\rho_*[\Theta]^{\vir,j}_{\mathbf{T}}=\frac{ [\M_{0,n,\beta}^{1/d,\epsilon}]^{\vir,j}_{\mathbf{T}}}{z-\psi_n}&+\sum_{n_1+n_2=n\atop \beta_1+\beta_2=\beta}\frac{[\M_{0,n_1+\bullet,\beta_1}^{1/d,\epsilon}\times\M_{0,n_2+\star,\beta_2}^{1/d,\epsilon}]^{\vir,j}_{\mathbf{T}}}{(z-\psi_\bullet)(-z-\psi_\star)}\\
&+z\sum_{\beta_1\leq1/\epsilon}(b_{\beta_1})_*\frac{\ev_{n + 1}^*(\nu_{\beta_1}^{\epsilon,j}(z)) \cap[\M_{0,n+1,\beta-\beta_1}^{1/d,\epsilon}]^{\vir,j}_{\mathbf{T}}}{-z-\psi_{n+1}},
\end{align*}
where the product in the second term is a divisor in $\M_{0,n,\beta}^{1/d,\epsilon}$, and its virtual class is determined by the virtual classes on each component and the pairing via the usual splitting property. The fact that the total contribution is regular at $z=0$ shows that the principal part of the first term is determined by the principal parts of the other terms, which are determined recursively in $(\beta,n)$.

We now turn our attention to the right-hand side of the theorem. Consider the sum of graph space classes
\begin{equation}\label{eq:relativeclass}
\sum_{\beta_1+\cdots+\beta_k=\beta}\frac{1}{k!} b_{\bbeta*}\left(\prod_{i = 1}^k \ev_{n + i}^*(\nu_{\beta_i}^{\epsilon,v}(-\psi_{n + i})) \cap[\mathcal{G}\M_{0,n+k,0}^{1/d,\infty}]^{\vir,j}_{\mathbf{T}} \right).
\end{equation}
Similar to the previous case, let $\Omega$ be the substack where the $n$th marked point lies over $\infty$ and at least one of the last $k$ marked points lies over $0$ (notice that $k\geq 1$ because $\beta>0$). The class \eqref{eq:relativeclass} restricted to this substack is again regular at $z=0$, and so is its pushforward. As in the previous case, the pushforward can be computed by localization and there are three types of fixed loci:
\begin{enumerate}
\item $\Omega_\infty$, where $\infty\in C_0$ is a smooth marked point,
\item $\Omega_{n_1,\beta_1|n_2,\beta_2}$, where the marked points split up in a stable way, and
\item $\Omega_{\beta_1}$, where $0\in C_0$ is a smooth marked point.
\end{enumerate}
The localization contribution of $\Omega_\infty$ is equal to
\[
\frac{\sum\limits_{\beta_1+\cdots+\beta_k=\beta}\frac{1}{k!} b_{\bbeta*} \left(\prod_{i = 1}^k \ev_{n + i}^*(\nu_{\beta_i}^{\epsilon,j}(-\psi_{n + i})) \cap[\M_{0,n+k,0}^{1/d,\infty}]^{\vir,j}_{\mathbf{T}} \right)}{z-\psi_n},
\]
and, by the induction hypothesis, the contributions of $\Omega_{n_1,\beta_1|n_2,\beta_2}$ and $\Omega_{\beta_1}$ are the same as the contributions of $\Theta_{n_1,\beta_1|n_2,\beta_2}$ and $\Theta_{\beta_1}$. Thus, the principal parts of the contributions of $\Omega_\infty$ and $\Theta_\infty$ are the same, finishing the induction step.

\end{proof}
\end{theorem}

As a result of the previous theorem, we obtain the following statement on the level of generating series.

\begin{corollary}
\label{cor:vertexJWC}
For any $j\in \{1,\ldots, N\}$, the twisted vertex $J$-functions satisfy the wall-crossing formula
\[J^{\epsilon, j}(q,\mathbf{t}(z),z) = J^{\infty, j}(\mathbf{t}(z)+z \mathbf{\one} + [J^{\epsilon,j}]_+(q,-z),z).\]
\end{corollary}

\begin{proof}
Integrate both sides of Theorem \ref{thm:vertexWC}.
\end{proof}

\section{Proofs of main theorems}
\label{sec:proofs}

In this section, we use the localization calculations of Section~\ref{sec:localization} and the vertex wall-crossing results of Section~\ref{sec:vertexWC} to prove the two main theorems.

\subsection{Proof of Theorem \ref{thm:1}}

The contents of this subsection are closely modeled on the work of Brown \cite{Brown} and Coates--Corti--Iritani--Tseng \cite{CCIT}, and they follow previous applications of these ideas to the hybrid model in \cite{ClaRo} and \cite{RR}. More specifically, in order to prove 
\[
J^{\epsilon}(q,z) = J^{\infty}(q,z \mathbf{\one} + [J^{\epsilon}]_+(q,-z),z),
\]
we characterize the right-hand side as an element of $\widetilde\H((z^{-1}))[[q]]$, then we show that the left-hand side satisfies this characterization.

For both the equivariant and non-equivariant settings, define
\[
\mathcal{V} := \widetilde{\mathcal{H}}_{(\mathbf{T})}((z^{-1}))[[q]],
\]
and consider the subset\footnote{In Givental's formalism, one gives $\mathcal{V}$ the structure of a symplectic vector space and proves that $\mathcal{L}$ is an overruled Lagrangian cone.  These properties can be proven in the current setting, using the fact that the genus-zero $\infty$-stable hybrid model correlators satisfy the string and dilaton equations and topological recursion relations.   However, these properties are not necessary for our purposes.}
\begin{equation}
\label{eq:L}
\mathcal{L}_{(\mathbf{T})} := \left\{\iota_*J^{\infty}(q,\mathbf{t},-z) \; | \; \mathbf{t}(z) \in \widetilde\H_{(\mathbf{T})}[z][[q]]\right\} \subseteq \mathcal{V}_{(\mathbf{T})},
\end{equation}
where $\mathbf{t}(z)$ satisfies $\mathbf{t}(z)|_{q=0}=0$, which ensures that the elements converge as power series in $q$, and $\iota_*:\widetilde\H^{\text{ct}}\rightarrow\widetilde\H$ is the injection
\[
H^j_{(m)}\mapsto H^j_{(m)} \frac{1}{d}\prod_{i \in F_m}\left(-\frac{w_i}{d}H_{(m)}\right).
\]
The reason for the map $\iota_*$ is that it allows one to write
\begin{align}\label{eq:iotaJ}
\iota_*J^{\epsilon}(q,-z) = -\frac{z}{d}\sum_{\beta \leq 1/\epsilon} q^{\beta}&\left(\frac{\prod_{i=1}^M\prod_{\substack{0 \leq b < \frac{w_i}{d}(\beta+1)\\ \langle b \rangle = \langle \frac{w_i}{d}(\beta+1) \rangle}} \left(bz - \frac{w_i}{d}H_{\left(\frac{\beta+1}{d}\right)}\right)}{\prod_{j=1}^N\prod_{\substack{0 < b \leq \beta \\ \langle b \rangle = 0}} \left(-bz + H_{\left(\frac{\beta+1}{d}\right)}\right)}\right)\\
\nonumber + &\sum_{\beta > 1/\epsilon} q^{\beta}\sum_{l=0}^{N-1} \left\langle \frac{H^l_{\left(-\frac{\beta+1}{d}\right)}}{-z-\psi_1}\right\rangle^{\epsilon}_{0,1,\beta} H^{N-l}_{\left(\frac{\beta+1}{d}\right)}\in\widetilde\H((z^{-1}))[[q]],
\end{align}
and one can safely include the terms $l>N-1-\left|F_{\left(\frac{\beta+1}{d} \right)}\right|$ in the second sum because the corresponding invariants vanish by Lemma \ref{lem:vanishing}.  The first summation in \eqref{eq:iotaJ} encodes the ``unstable terms" and the second summation the ``stable terms." 

Analogously, for each $\mathbf{T}$-fixed point $P_j \in Z$, let
\[
\mathcal{V}^j := \widetilde\H_{j}((z))[[q]],
\]
and let
\begin{equation}
\label{eq:vertexL}
\mathcal{L}^{j} = \left\{\iota_*J^{\infty, j}(q,\mathbf{t},-z) \; | \; \mathbf{t}(z) \in \widetilde\H_{j}[z][[q]]\right\} \subseteq \mathcal{V}^j,
\end{equation}
where $\mathbf{t}(z)$ satisfies $\mathbf{t}(z)|_{q=0}=0$ and, restricting to the fixed point, $\iota_*:\widetilde\H_{j}\rightarrow \widetilde\H_{j}$ becomes
\[
\one_{(m)}^j\mapsto \one_{(m)}^j \frac{1}{d}\prod_{i\in F_m}\left(-\frac{w_i\alpha_j}{d}\right).
\]

\begin{remark}
It is essential and worth pointing out that $\mathcal{L}$ consists of Laurent series in $z^{-1}$ while $\mathcal{L}^{j}$ consists of Laurent series in $z$.
\end{remark}

In order to prove Theorem~\ref{thm:1}, it suffices to prove that $\iota_*J^{\epsilon}(q,-z) \in \mathcal{L}_{\mathbf{T}}$ for some equivariant lift of $\iota_*J^{\epsilon}(q,-z) $.  Indeed, by \eqref{eq:L}, this will prove that there exists some $\mathbf{t} \in \widetilde{\mathcal{H}}_\mathbf{T}[z][[q]]$ for which $J^{\epsilon}(q,z) = J^{\infty}(q,\mathbf{t},z)$.  The specific choice of $\mathbf{t}$ is determined by the fact that
\[J^{\infty}(q,\mathbf{t},z) = z\one+ \mathbf{t}(-z) + O(z^{-1}),\]
so taking the part of the equation $J^{\epsilon}(q,z) = J^{\infty}(q,\mathbf{t},z)$ with non-negative powers of $z$ yields
\[
\mathbf{t}(z) = z\one+ [J^{\epsilon}]_+(q,-z).
\]
Taking the non-equivariant limit proves Theorem~\ref{thm:1}.

The strategy for proving that $\iota_*J^{\epsilon}(q,-z) \in \mathcal{L}_{\mathbf{T}}$ is to prove a characterization of elements of $\mathcal{L}_{\mathbf{T}}$.  We make use of the following notation.  If $\mathbf{f} \in \mathcal{V}_{\mathbf{T}}$, then for each $j \in \{1,\ldots, N\}$, we denote by $\mathbf{f}_j$ the image of $\mathbf{f}$ under the restriction map
\[
\widetilde{\mathcal{H}}_{\mathbf{T}}\rightarrow \widetilde\H_{j}.
\]
For $m \in \{0,\frac{1}{d}, \ldots, \frac{d-1}{d}\}$, we denote by $\mathbf{f}_{j,m}$ the coefficient of $\iota_*(\one_{(-m)}^j)^\vee=\prod_{k\neq j}(\alpha_j-\alpha_k)\one_{(m)}^j$ in $\mathbf{f}$.

For each $m,m' \in \left\{0,\frac{1}{d},\dots,\frac{d-1}{d}\right\}$, we set
\[E^{m,m'} := \left\{\beta \in \Z_{>0}\; \left| \; \frac{\beta}{d} - m - m' \in \Z \right.\right\},\]
so that $E^{m,m'}$ is the set of possible degrees $\beta_e$ in a localization graph for which $e$ is an edge adjacent to vertices $v$ and $v'$ with $m_{e,v} = m$ and $m_{e,v'} = m'$.  For each $\beta \in E^{m,m'}$, we define the {\it recursive term}
\[
\text{RC}^{m,m'}_{j,j'}(\beta) := \frac{1}{\beta} \cdot \frac{\prod_{i=1}^M \prod_{\substack{0 \leq b < \frac{\beta w_i}{d}\\ \langle b \rangle = \langle w_im \rangle}}\left(\frac{b}{\beta}(\alpha_{j}-\alpha_{j'}) - \frac{w_i}{d}\alpha_j\right)}{\prod_{k=1}^N \prod'_{\substack{0 \leq b < \beta\\ \langle b \rangle =0 }} \left( \frac{b}{\beta}(\alpha_{j'} - \alpha_{j}) + \alpha_j - \alpha_{k}\right)}.
\]

With this notation established, elements of $\mathcal{L}_{\mathbf{T}}$ are characterized as follows.

\begin{proposition}
\label{prop:conechar}
An element $\mathbf{f} \in \mathcal{V}$ lies in $\mathcal{L}_{\mathbf{T}}$ if and only if the following are satisfied:
\begin{enumerate}
\item For each $j$ and $m$, the restriction $\mathbf{f}_{j,m}$ lies in $\C(z,\alpha_1, \ldots, \alpha_N)[[q]]$ and, as a rational function of $z$, each $q^{\beta}$-coefficient of $\mathbf{f}_{j,m}$ is regular except for possible poles at $z=0$, $z= \infty$, and $z = (\alpha_{j'}-\alpha_j)/\beta$ with $\beta \in E^{m,m'}$ for some $j',m'$.
\item For each $j \neq j'$, each $m,m'$, and each $\beta \in E^{m,m'}$, we have
\[
\text{Res}_{z=\frac{\alpha_{j}-\alpha_{j'}}{\beta}} \mathbf{f}_{j,m} = -q^\beta \cdot \text{RC}^{m,m'}_{j,j'}(\beta) \cdot \mathbf{f}_{j',-m'}\big|_{z=\frac{\alpha_{j}-\alpha_{j'}}{\beta}}.
\]
\item The Laurent expansion of each $\mathbf{f}_j$ at $z=0$ lies in $\mathcal{L}^{j}$.
\end{enumerate}
\begin{proof}
The proof is similar to that in \cite{ClaRo}, which differs from the current setting only in the twist of the universal bundle at broad marked points. For completeness, we sketch the main ideas, which will be expanded upon further in the proof of Theorem \ref{thm:1}, where both the stable and unstable cases are treated.

Assume $\mathbf{f} \in \mathcal{L}_{\mathbf{T}}$. Then $\mathbf{f}=\iota_*J^{\infty}(q,\mathbf{t},-z)$ for some $\mathbf{t}(z)$ and it follows from the localization formula that 
\begin{align}\label{eq:cone}
\nonumber \mathbf{f}_j=\iota_*&\Bigg(z\one_{(1/d)}^j+\mathbf{t}_j(z)+\sum_{{j'\neq j}\atop {m,m'\in\left\{0,\frac{1}{d},\dots,\frac{d-1}{d}\right\} \atop \beta\in E^{m,m'}}}\frac{T^{m,m'}_{j,j'}}{-z+\frac{\alpha_j-\alpha_{j'}}{\beta}}\\
&+\sum_{n\geq 2 \atop k\in\{0,\frac{1}{d},\dots,\frac{d-1}{d}\}}\frac{1}{n!}\left\langle\left(\mathbf{t}_j(\psi)+\sum_{j',m,m',\beta}\frac{T^{m,m'}_{j,j'}}{-\psi+\frac{\alpha_j-\alpha_{j'}}{\beta}}\right)^n\frac{\one_{(k)}^j}{-z-\psi_{n+1}}\right\rangle_{0,n+1}^{\infty,j}(\one_{(k)}^j)^\vee\Bigg).
\end{align}
Indeed, the term $\frac{T^{m,m'}_{j,j'}}{-z+\frac{\alpha_j-\alpha_{j'}}{\beta}}\in \widetilde\H_j(z)[[q]]$ is the sum of all localization contributions from graphs where the last marked point is on an unstable vertex of valence two with adjacent edge $e$ having opposite vertex $v'$ with $j_{v'}=j'$, $m_{(e,v')}=m'$, and $m_{(e,v)}=m$. The second line of \eqref{eq:cone} collects all localization contributions where the last marked point is on a stable vertex. Properties (1) and (3) are observed from \eqref{eq:cone}. The recursion in property (2) reflects the removal of the edge $e$ from $\frac{T^{m,m'}_{j'}}{-z+\frac{\alpha_j-\alpha_{j'}}{\beta}}$.

Conversely, suppose $\mathbf{f}$ satisfies properties (1) and (3). Then, by a partial fractions decomposition, $\mathbf{f}$ can be written in the form \eqref{eq:cone} where the terms $\mathbf{t}_j(z)$ and $T^{m,m'}_{j,j'}$ are undetermined power series in $q$. Since property (2) recursively (in $q$) determines $T^{m,m'}_{j,j'}$, we see that properties (1), (2), and (3) determine $\mathbf{f}$ up to $\mathbf{t}(z)$. It follows that $\mathbf{f}=\iota_*J^\infty(q,\mathbf{t},-z)$, because $\iota_*J^\infty(q,\mathbf{t},-z)$ also satisfies properties (1), (2), and (3), and the regular parts when expanded as Laurent series in $z^{-1}$ of the two sides are both equal to $\mathbf{t}(z)$. Thus, $\mathbf{f}\in\mathcal{L}_{\mathbf{T}}$.
\end{proof}
\end{proposition}

Equipped with Proposition~\ref{prop:conechar}, all that remains in order to prove Theorem~\ref{thm:1} is to verify that $\iota_*J^{\epsilon}(q,-z)$ satisfies conditions (1), (2), and (3).

\begin{proof}[Proof of Theorem~\ref{thm:1}]
Throughout, we write $\mathbf{f}=\iota_*J^{\epsilon}(q,-z)$ as in \eqref{eq:iotaJ}.   We first prove that $\mathbf{f}$ satisfies condition (1) of Proposition~\ref{prop:conechar}.  The contribution of the unstable terms to $\mathbf{f}_{j,m}$ is
\begin{equation}
\label{eq:unstjm}
-\frac{z}{d}\sum_{\substack{\beta \leq 1/\epsilon \\ \left\langle \frac{\beta+1}{d}\right\rangle = m}}q^{\beta} \left(\frac{\prod_{i=1}^M\prod_{\substack{0 \leq b < \frac{w_i}{d}(\beta+1)\\ \langle b \rangle = \langle w_i m \rangle}} (bz - w_i \frac{\alpha_j}{d})}{\prod_{k=1}^N\prod'_{\substack{0 \leq b \leq \beta \\ \langle b \rangle = 0}} (-bz + \alpha_j-\alpha_k)}\right),
\end{equation}
which is manifestly a rational function of $z$ with the prescribed poles.  The contribution of the stable terms to $\mathbf{f}_{j,m}$ can be calculated by localization, with all contributing graphs having the marked point on a vertex $v$ with $j_v = j$. As in \eqref{eq:cone}, there are two types of graphs, depending on whether $v$ is unstable or stable. If $v$ is unstable with adjacent edge $e$, then the graph contribution is rational in $z$ with pole at $z=\frac{\alpha_{j}-\alpha_{j'}}{\beta_e}$. If $v$ is stable, then the graph contribution is polynomial in $z^{-1}$ (because $\psi$ is nilpotent on stable vertices). This verifies condition (1).


We next prove that $\mathbf{f}$ satisfies condition (2).  We begin with unstable terms \eqref{eq:unstjm}, for which one can calculate directly that the residue at $z= \frac{\alpha_{j} - \alpha_{j'}}{\beta_e}$ of $\mathbf{f}_{j,m}$ equals
\begin{equation}
\label{eq:residue}
\frac{\alpha_{j} - \alpha_{j'}}{d\beta_e^2} \sum_{\substack{\beta_e\leq\beta \leq 1/\epsilon \\ \left\langle \frac{\beta+1}{d}\right\rangle = m}}q^{\beta} \left(\frac{\prod_{i=1}^M\prod_{\substack{0 \leq b < \frac{w_i}{d}(\beta+1)\\ \langle b \rangle = \langle w_i m \rangle}} (b\frac{\alpha_{j}-\alpha_{j'}}{\beta_e} - w_i \frac{\alpha_j}{d})}{\prod_{k=1}^N\prod'_{\substack{0 \leq b \leq \beta \\ \langle b \rangle = 0}} (b\frac{\alpha_{j'}-\alpha_{j}}{\beta_e} + \alpha_j-\alpha_k)}\right).
\end{equation}
The evaluation of the unstable terms in $\mathbf{f}_{j',-m'}$ at $z=\frac{\alpha_{j}-\alpha_{j'}}{\beta_e}$ equals
\begin{equation}
\label{eq:evaluation}
\frac{\alpha_{j'} - \alpha_{j}}{d\beta_e} \sum_{\substack{\beta_{v'} \leq 1/\epsilon \\ \left\langle \frac{\beta_{v'}+1}{d}\right\rangle = -m'}}q^{\beta_{v'}} \left(\frac{\prod_{i=1}^M\prod_{\substack{0 \leq b < \frac{w_i}{d}(\beta_{v'}+1)\\ \langle b \rangle = \langle -w_i m' \rangle}} (b\frac{\alpha_{j}-\alpha_{j'}}{\beta_e} - w_i \frac{\alpha_{j'}}{d})}{\prod_{k=1}^N\prod'_{\substack{0 \leq b \leq \beta_{v'} \\ \langle b \rangle = 0}} (b\frac{\alpha_{j'}-\alpha_{j}}{\beta_e} + \alpha_{j'}-\alpha_k)}\right).
\end{equation}
By shifting the index $\beta_{v'}$ in \eqref{eq:evaluation} to $\beta=\beta_e+\beta_{v'}$, one checks directly that
\[
\eqref{eq:residue} = -q^{\beta_e} \text{RC}^{m,m'}_{j,j'}(\beta_e) \cdot \eqref{eq:evaluation} \mod \{q^\beta \; | \;\beta>1/\epsilon\}.
\]
The right-hand side of this equation has nontrivial coefficients of $q^\beta$ with $1/\epsilon<\beta\leq 1/\epsilon+\beta_e$. One checks that these correspond to the stable contributions to the residue at $z= \frac{\alpha_{j} - \alpha_{j'}}{\beta_e}$ of $\mathbf{f}_{j,m}$ coming from graphs with a single edge $e$ connecting two unstable vertices $v$ and $v'$ with the marked point on $v$.

For the remaining stable terms, the verification of condition (2) is again by localization, where contributing graphs have nonzero residue at $z= (\alpha_{j} - \alpha_{j'})/\beta_e$ only if the marked point lies on an unstable vertex $v$ with $j_v = j$, such that the unique edge $e$ adjacent to $v$ has degree $\beta_e$ and meets the rest of the graph at a vertex $v'$ with $j_{v'} = j'$.  The contribution of such a graph $\Gamma$ to the correlator
\[
\left\langle \frac{\mathbf{1}^j_{(-m)}}{-z-\psi_1}\right\rangle^{\epsilon}_{0,1,\beta}
\]
can be expressed as
\[\text{Contr}_{\Gamma} = \frac{1}{-z+\frac{\alpha_j-\alpha_{j'}}{\beta_e}}\cdot \text{RC}^{m,m'}_{j,j'}(\beta_e) \cdot \text{Contr}_{\Gamma'},\]
where $\Gamma'$ is the graph obtained from $\Gamma$ by omitting the edge $e$ and $\text{Contr}_{\Gamma'}$ is the contribution of $\Gamma'$ to the correlator
\[\left\langle \frac{\mathbf{1}^{j'}_{(m')}}{\frac{\alpha_j-\alpha_{j'}}{\beta_e}-\psi_1}\right\rangle^{\epsilon}_{0,1,\beta-\beta_e}.\]
Summing over all possible graphs $\Gamma$ completes the verification of condition (2).

Finally, we prove that $\mathbf{f}$ satisfies condition (3).  Let 
\[
\tau_j(z)=\sum_\Gamma\text{Contr}_\Gamma,
\]
where, as above, $\Gamma$ is a graph where the marked point lies on an unstable vertex $v$ with $j_v = j$ and $\text{Contr}_\Gamma$ denotes the contribution to $\mathbf{f}_j$. The sum of all contributions to $f_j$ from graphs where the marked point is on a stable vertex $v$ with $j_v=j$ can then be written as
\[
\sum_{\substack{\beta > 1/\epsilon\\ n,m}}\frac{q^{\beta}}{n!} \left\langle \tau_j(\psi_1) \cdots \tau_j(\psi_n) \cdot \frac{\mathbf{1}^j_{(-m)}}{-z-\psi_{n+1}} \right\rangle^{\epsilon,j}_{0,n+1,\beta} \iota_*(\mathbf{1}^j_{(-m)})^\vee.
\]
The unstable contributions to $\mathbf{f}_j$, on the other hand, are exactly the unstable contributions to $\iota_*J^{\epsilon,\text{tw}_j}(q,-z)$.  It follows that
\[
\mathbf{f} = \iota_*J^{\epsilon,j}(q, \tau_j(z), -z),
\]
which, by Corollary \ref{cor:vertexJWC}, equals $\iota_*J^{\infty,j}(q,\tau_j(z) + z \mathbf{1} + [J^{\epsilon,\text{tw}_j}]_+(q,-z),-z)$ and hence lies on $\mathcal{L}^{\infty,j}$.  This completes the verification of condition (3) and hence the proof of Theorem~\ref{thm:1}.
\end{proof}

\subsection{Proof of Theorem~\ref{thm:2}}

We now prove the wall-crossing theorem for virtual cycles:
\[
\sum_{\beta} q^\beta  [Z^{\epsilon}_{0,n,\beta}]^{\vir} =\sum_{\beta_0, \beta_1, \ldots, \beta_k} \frac{q^{\beta_0}}{k!} b_{\bbeta*}c_* \left(\prod_{i = 1}^k q^{\beta_i} \ev_{n + i}^*(\mu_{\beta_i}^\epsilon(-\psi_{n + i})) \cap [Z^{\infty}_{0,n+k,\beta_0}]^{\vir} \right),
\]
which is an application of the localization formula on both sides. 

\begin{proof}[Proof of Theorem~\ref{thm:2}]
By localization on the space $Z^{\epsilon}_{0,n,\beta}$ and the calculations of Section \ref{subsec:loccontr}, the left-hand side of \eqref{eq:mainWC} can be expressed as
\begin{equation}
\label{eq:LHS}
\sum_{\Gamma}\text{Contr}_\Gamma^{\text{LHS}}=\sum_{\Gamma} \prod_{v \in V(\Gamma)} q^{\beta_v}\text{Contr}_{\Gamma}(v)\prod_{(v,e) \in F(\Gamma)}\text{Contr}_\Gamma(v,e) \prod_{e \in E(\Gamma)} q^{\beta_e}\text{Contr}_{\Gamma}(e),
\end{equation}
where the sum is over localization graphs for the moduli spaces $Z^{\epsilon}_{0,n,\beta}$ for all $\beta$.  Let $v$ be a vertex of a localization graph $\Gamma$, and for convenience, set
\[\text{Contr}_{\Gamma}^E(v) :=\prod_{e \in E_v} \frac{\ev_e^*\left(\one^{j_v}_{(m_{(v,e)})}\right)}{\frac{\alpha_{j_{v}}-\alpha_{j_{v'}}}{\beta_e} - \psi_{(v,e)}}.\]
If $v$ is a stable vertex, then equation \eqref{eq:stablevertex} and Theorem~\ref{thm:vertexWC} together imply that
\begin{align*}
\text{Contr}_{\Gamma}(v) &=\text{Contr}_{\Gamma}^E(v) \cap [\M^{1/d,\epsilon}_{0,\text{val}(v), \beta_v}]^{\vir,{j_v}}_{\mathbf{T}}\\
&=\text{Contr}_{\Gamma}^E(v) \cap \hspace{-0.6cm} \sum_{\substack{k\\ \beta_1 + \cdots + \beta_k = \beta_v}} \hspace{-0.6cm}\frac{1}{k!} b_{\vec{\beta}*} c_*\left( \prod_{i=1}^k \ev_{\text{val}(v)+i}^*(\nu_{\beta_i}^{\epsilon,j_v}(-\psi_{\text{val}(v)+i})) \cap [\M^{1/d,\infty}_{0,\text{val}(v)+k,0}]^{\vir,{j_v}}_{\mathbf{T}} \right).
\end{align*}
By Remark \ref{rmk:unstable}, this equation holds even when $v$ is an unstable vertex. The condition $n>0$ is important here, as it implies that $\text{val}(v)>0$ at every vertex $v$, which in turn allows us to apply Theorem \ref{thm:vertexWC} at every vertex.

Let $i_{j}: \{P_j\} \hookrightarrow X$ be the inclusion of the $j$th $\mathbf{T}$-fixed point.  Then $J^{\epsilon,{j_v}}$ has the same unstable terms as $i_{j_v}^*J^{\epsilon}=J^{\epsilon}_{j_v}$, and hence 
\[\nu^{\epsilon,j_v}_{\beta_i}(z) = i_{j_v}^*\mu^{\epsilon}_{\beta_i}(z).\]
It follows that $\text{Contr}_{\Gamma}(v)$ equals
\begin{align*}
&\text{Contr}_{\Gamma}^E(v) \cap \hspace{-0.5cm} \sum_{\substack{k \\ \beta_1 + \cdots + \beta_k = \beta_v}} \hspace{-0.5cm}\frac{1}{k!} b_{\vec{\beta}*} c_*\left( \prod_{i=1}^k \ev_{\text{val}(v)+i}^*(i_{j_v}^*\mu^{\epsilon}_{\beta_i}(-\psi_{\text{val}(v)+i})) \cap [\M^{1/d,\infty}_{0,\text{val}(v)+k,0}]^{\vir,{j_v}}_{\mathbf{T}} \right)\\
=&\text{Contr}_{\Gamma}^E(v) \cap \hspace{-0.5cm} \sum_{\substack{k \\ \beta_1 + \cdots + \beta_k = \beta_v}} \hspace{-0.5cm} \frac{1}{k!} b_{\vec{\beta}*} c_*\left( \prod_{i=1}^k \ev_{\text{val}(v)+i}^*\bigg([q^{\beta_i}][J^{\infty}_{j_v}]_+\Big(q,[J^{\epsilon}]_+(q,-z) + z\one,-\psi_{\text{val}(v)+i}\Big)\bigg)\right.\\
&\hspace{2cm} \cap[\M^{1/d,\infty}_{0,\text{val}(v)+k,0}]^{\vir,j_v}_{\mathbf{T}} \bigg),
\end{align*}
where the equality is an application of Theorem~\ref{thm:1} and $[q^{\beta_i}]$ again denotes the coefficient of $q^{\beta_i}$ in a power series in $q$.

Now, we have
\[
[J^{\infty}_{j_v}]_+(q, \mathbf{t}(z),z) = z\one + \mathbf{t}(-z) + \sum_{\Lambda} \text{Contr}_{\Lambda}(J^{\infty}(q,\mathbf{t}(z),z)),
\]
where the sum is over localization graphs $\Lambda$ for the moduli spaces $Z^{\infty}_{0,n+1,\beta}$ such that the last marked point lies on an unstable vertex $w$ with $j_w = j_v$; here $\text{Contr}_{\Lambda}(J^{\infty}(q,\mathbf{t}(z),z)$ denotes the contribution of $\Lambda$ to the localization expression for a stable term of $J^{\infty}(q,\mathbf{t}(z),z)$.  Thus,
\begin{align}
\label{eq:Jjv}
\nonumber &[q^{\beta_i}][J^{\infty}_{j_v}]_+\Big(q,[J^{\epsilon}]_+(q,-z) + z\one,-\psi_{\text{val}(v)+i})\\
= &i_{j_v}^*\mu^{\epsilon}_{\beta_i}(-\psi_{\text{val}(v)+i}) + \sum_{\Lambda} [q^{\beta_i}]\text{Contr}_{\Lambda}(J^{\infty}(q,[J^{\epsilon}]_+(q,-z) + z\one,-\psi_{\text{val(v)}+i}))=:T_{j_v}^{\beta_i},
\end{align}
where we can think of $T_{j_v}^{\beta_i}$ as the $q^{\beta_i}$-coefficient of a `tail' emanating from the vertex $v$. We then write
\begin{align*}
\text{Contr}_{\Gamma}(v) &=\text{Contr}_{\Gamma}^{E}(v) \; \cap\hspace{-.4cm}\sum_{\substack{ k \\ \beta_1 + \cdots + \beta_k = \beta_v}} \hspace{-0.5cm} \frac{1}{k!} b_{\vec{\beta}*} c_*\left( \prod_{i=1}^k \ev_{n+i}^*\left(T_{j_v}^{\beta_i} \right)\cap[\M^{1/d,\infty}_{0,\text{val}(v)+k,0}]^{\vir,j_v}_{\mathbf{T}} \right).
\end{align*}
By the localization formula,
\begin{align*}
\sum_{\substack{ k \\ \beta_1 + \cdots + \beta_k = \beta_v}} \hspace{-0.5cm} \frac{1}{k!}&\prod_{i=1}^k \ev_{n+i}^*\left(T_{j_v}^{\beta_i} \right)\cap[\M^{1/d,\infty}_{0,\text{val}(v)+k,0}]^{\vir,j_v}_{\mathbf{T}}\\
&=\hspace{-0.5cm}\sum_{\substack{ m \\ \beta_1 + \cdots + \beta_m \leq \beta_v}}\hspace{-0.5cm}\frac{1}{m!}\sum_{\Omega}\text{Contr}_{\Omega}\left(\prod_{i = 1}^m \ev_{\text{val}(v) + i}^*(\mu_{\beta_i}^\epsilon(-\psi_{\text{val}(v) + i}))\right),
\end{align*}
where the second sum in the right-hand side is over localization graphs $\Omega$ for the moduli spaces $Z^\infty_{0,\text{val}(v)+m,\beta_v-\sum_{i=1}^m\beta_i}$ such that (at least) the first $\text{val}(v)$ marked points lie on a distinguished vertex $w$ with $j_w=j_v$, and such that each of the $k$ connected components $T$ in $\Omega\setminus\{w\}$ satisfies $\beta(T)+\sum_{i\in T}\beta_i\leq 1/\epsilon$, so that the entire fixed locus $F_\Omega$ maps to $\M^{1/d,\epsilon}_{0,\text{val}(v), \beta_v}$ upon applying the map $b_{\vec\beta}\circ c$. Furthermore,
\[
\text{Contr}_\Omega\left(\prod_{i = 1}^m \ev_{\text{val}(v) + i}^*(\mu_{\beta_i}^\epsilon(-\psi_{\text{val}(v) + i}))\right)\in H_*(\M^{1/d,\epsilon}_{0,\text{val}(v)+k, 0})
\]
denotes the result of taking the localization contribution of $\Omega$ to the class
\[
\prod_{i = 1}^m \ev_{\text{val}(v) + i}^*(\mu_{\beta_i}^\epsilon(-\psi_{\text{val}(v) + i})) \cap [Z^{\infty}_{0,\text{val}(v)+m,\beta_v-\sum_{i=1}^m\beta_i}]^{\vir}
\]
and integrating along all vertex moduli spaces except the distinguished vertex. Since integrating along all of the vertex moduli except the distinguished one and then replacing the attaching node with a basepoint is the same as applying the map $b_{\vec\beta*}c_*$, this implies that 
\[
\text{Contr}_{\Gamma}(v)=\text{Contr}_{\Gamma}^{E}(v) \; \cap \hspace{-.5cm}\sum_{\substack{ m \\ \beta_1 + \cdots + \beta_m \leq \beta_v}}\hspace{-0.5cm}\frac{1}{m!}b_{\vec\beta*}c_*\sum_{\Omega}\text{Contr}_{\Omega}\left(\prod_{i = 1}^m \ev_{\text{val}(v) + i}^*(\mu_{\beta_i}^\epsilon(-\psi_{\text{val}(v) + i}))\right).
\]
Applying this procedure at each vertex of $\Gamma$, it follows that
\[
\text{Contr}_\Gamma^{\text{LHS}}= \hspace{-.5cm}\sum_{\substack{ m \\ \beta_1 + \cdots + \beta_m \leq \sum \beta_v}}\hspace{-0.5cm}\frac{1}{m!}b_{\vec\beta*}c_*\sum_{\Omega}\text{Contr}_{\Omega}^{\text{RHS}}
\]
where the second sum is over localization graphs $\Omega$ for the moduli spaces $Z^\infty_{0,n+m,\beta}$ such that $F_\Omega$ maps to $F_\Gamma$ upon applying the map $b_{\vec\beta}\circ c$, and $\text{Contr}_{\Omega}^{\text{RHS}}$ denotes the contribution to 
\[
\prod_{i = 1}^m \ev_{n + i}^*(\mu_{\beta_i}^\epsilon(-\psi_{n + i})) \cap [Z^{\infty}_{0,n+m,\beta_0}]^{\vir}.
\]
Summing over all localization graphs $\Gamma$ on the left-hand side is equivalent to summing over all localization graphs $\Omega$ on the right-hand side, completing the proof that the two sides of \eqref{eq:mainWC} are equal.
\end{proof}

\bibliographystyle{abbrv}
\bibliography{biblio}

\end{document}